\documentclass[11pt, reqno]{amsart}
\usepackage{amscd,amsfonts,amsmath,amssymb,amstext,amsthm}
\usepackage[all]{xy}
\usepackage{xcolor}
\usepackage{enumerate}
\usepackage{verbatim}
\usepackage{enumitem}

\makeatletter

\theoremstyle{plain}
\newtheorem{theorem} {Theorem} [section]
\newtheorem{lemma} [theorem]{Lemma}
\newtheorem{proposition}[theorem]{Proposition}
\newtheorem{corollary} [theorem]{Corollary}
\newtheorem{question}[theorem]{Question}
\theoremstyle{definition}

\newtheorem{remark}[theorem]{Remark}
\numberwithin{equation}{section}

\title[Weighted $L^2$ Holomorphic functions on ball-fiber bundles ]{Weighted $L^2$ Holomorphic functions on ball-fiber bundles over compact K\"ahler manifolds}
\date\today
\author{Seungjae Lee and Aeryeong Seo}
\address{Department of Mathematics,
Kyungpook National University,
Daegu 41566, Republic of Korea}%
\email{aeryeong.seo@knu.ac.kr}
\address{Center for Complex Geometry, Institute for Basic Science, Daejeon, 34126, Republic of Korea}
\email{seungjae@ibs.re.kr}


\begin{document}
\maketitle
\begin{abstract}
Let $\widetilde{M}$ be a complex manifold, $\Gamma$ be a torsion-free cocompact lattice of $\text{Aut}(\widetilde{M})$ and $\rho\colon\Gamma\to SU(N,1)$ be a representation.
Suppose that there exists a $\rho$-equivariant totally geodesic isometric holomorphic { embedding} $\imath\colon \widetilde M\to\mathbb B^N$. Let
$M:=\widetilde M/\Gamma$ and $\Sigma:=\mathbb B^N/\rho(\Gamma)$.
In this paper, we investigate a relation between weighted $L^2$ holomorphic functions on the fiber bundle $\Omega:=M\times_\rho\mathbb B^N$ and the holomorphic sections of the pull-back bundle $\imath^*(S^mT^*_\Sigma)$ over $M$.  
In particular, $A^2_\alpha(\Omega)$ has infinite dimension for any $\alpha>-1$ and if $n<N$, then $A^2_{-1}(\Omega)$ also has the same property.
As an application, if $\Gamma$ is a torsion-free cocompact lattice in $SU(n,1)$, $n\geq 2$, and $\rho\colon \Gamma\to SU(N,1)$ is a maximal representation, then for any $\alpha>-1$, $A^2_\alpha(\mathbb B^n\times_{\rho} \mathbb B^N)$ has infinite dimension. If $n<N$, then $A_{-1}^2(\mathbb B^n\times_{\rho} \mathbb B^N)$ also has the same property.
\end{abstract}

\section{Introduction}
For a complex manifold $X$, denote by $\text{Aut(X)}$ the set of holomorphic diffeomorphisms of $X$ onto itself and denote by $S^m T^*_X$ the $m$-th symmetric power of the holomorphic cotangent bundle of $X$.
For a holomorphic fiber bundle $E\to X$ over $X$ and a holomorphic map $f\colon Y\to X$,
we denote by $f^*(E)$ the pull-back bundle of $E$ over $Y$ by $f$. 
Let $\mathbb B^n := \{z\in \mathbb C^n : |z|<1\}$ be the $n$-dimensional unit ball.
For a lattice $\Gamma$ in $\textup{Aut}(X)$ and a homomorphism $\rho\colon\Gamma\to\textup{Aut}(\mathbb B^N)$, we say that a map $f\colon X\to\mathbb B^N$ is $\rho$-equivariant if for any $\gamma\in\Gamma$ and $z\in X$, $f(\gamma z) = \rho(\gamma)f(z)$ holds.

Our primary result of this paper is
\begin{theorem}\label{main}
Let $\widetilde{M}$ be a complex manifold, $\Gamma$ be a torsion-free cocompact lattice of $\text{Aut}(\widetilde{M})$ and $\rho\colon\Gamma\to SU(N,1)$ be a representation.
Suppose that there exists a $\rho$-equivariant totally geodesic isometric holomorphic embedding $\imath\colon \widetilde M\to\mathbb B^N$. Let
$M:=\widetilde M/\Gamma$ and $\Sigma:=\mathbb B^N/\rho(\Gamma)$.
Let $\Omega:=M\times_\rho \mathbb B^N$ be a holomorphic $\mathbb B^N$-fiber bundle over $M$ where any $\gamma\in \Gamma$ acts on $\widetilde M\times \mathbb B^N$ by $(\zeta,w)\mapsto (\gamma \zeta, \rho(\gamma) w)$. Then there exists an injective linear map
\begin{equation}\nonumber
\Phi: \bigoplus_{m=0}^{\infty} H^0 (M, \imath^* (S^m T_{\Sigma}^*)) \rightarrow
\begin{cases}
\displaystyle\bigcap_{\alpha>-1} A^2_\alpha (\Omega) \subset \mathcal{O} (\Omega) & \text{ if } n=N,\\
\displaystyle\bigcap_{\alpha\geq -1} A^2_\alpha (\Omega) \subset \mathcal{O} (\Omega)  & \text{ if } n< N,\\
\end{cases}
\end{equation}
which has a dense image in $\mathcal{O}(\Omega)$ equipped with the compact open topology. In particular, $\dim A^2_{\alpha} (\Omega) = \infty$ if $\alpha> -1$ and $A_{-1}^2 (\Omega) = \bigcap_{\alpha \geq -1} A_\alpha^2(\Omega)$ with $\dim A^2_{-1} (\Omega) = \infty$ if $n < N$.
\end{theorem}
We remark that under the condition of Theorem~\ref{main}, $\rho(\Gamma)$ acts on $\mathbb B^N$ properly discontinuously and hence $\Sigma$ is a complex manifold.
Let $\Gamma\subset SU(n,1)$ be a cocompact lattice and $\rho\colon \Gamma\to SU(N,1)$ be a homomorphism.
Denote by $\omega_n$ and $\omega_N$ the K\"ahler forms of the Bergman metrics of $\mathbb B^n$ and $\mathbb B^N$ respectively.
Let $f\colon \mathbb B^n\to \mathbb B^N$ be any smooth $\rho$-equivariant map and $[\rho^* \omega_N]:=[f^* \omega_N]\subset H^2_{dR}(\mathbb B^n/\Gamma)$ be the de Rham class of $f^*\omega_N$ which only depends on $\rho$.
The Toledo invariant $\tau(\rho)$ of $\rho$ is defined by
$$\tau(\rho):=
\frac{1}{n!}\int_{\mathbb B^n/\Gamma} \rho^* \omega_N\wedge \omega^{n-1}_{\mathbb B^n/\Gamma}
$$
and it satisfies the Milnor-Wood inequality
\begin{equation}\label{Toledo invariant}
|\tau(\rho)|\leq \text{Vol}(\mathbb B^n/\Gamma)
\end{equation}
under suitable normalizations of the metrics.
One says that $\rho$ is a {\it maximal representation} if the equality holds in \eqref{Toledo invariant}. In \cite{Corlette}, Corlette showed that if $\rho$ is a maximal representation with $n\geq 2$, then there exists a totally geodesic holomorphic $\rho$-equivariant embedding of $\mathbb B^n$ into $\mathbb B^N$. By Theorem~\ref{main}, we have
\begin{corollary}
Let $\Gamma\subset SU(n,1)$, $n\geq 2$, be a cocompact lattice and $\rho\colon \Gamma\to SU(N,1)$ be a maximal representation.
Let $\Omega:= M\times_\rho \mathbb B^N$ be a holomorphic $\mathbb B^N$-fiber bundle over a complex hyperbolic space form $M:=\mathbb B^n/\Gamma$. Then for each $\alpha>-1$, the dimension of $A_\alpha^2(\Omega)$ is infinite and $\bigcap_{\alpha>-1} A_\alpha^2(\Omega)$ is dense in $\mathcal O(\Omega)$ equipped with the compact open topology. Moreover if $n<N$, then the dimension of $A_\alpha^2(\Omega)$ is infinite for each $\alpha\geq-1$ and $A^2_{-1} (\Omega) = \bigcap_{\alpha\geq -1} A_\alpha^2(\Omega)$ is dense in $\mathcal O(\Omega)$ equipped with the compact open topology.
\end{corollary}

Theorem \ref{main} is motivated by the following question.
\begin{question}
Does any unit ball fiber bundle over a compact K\"ahler manifold
admit nonconstant weighted $L^p$ ($1\leq p\leq \infty$) holomorphic functions?
\end{question}
Remark that any $\mathbb B^N$-fiber bundle over a compact K\"ahler manifold admits $C^\infty$ plurisubharmonic exhaustion function if either $n=1$ (\cite{Diederich_Ohsawa}) or $n\geq 2$ and $\rho$ is reductive (\cite{Seo}).

Whether there exist holomorphic functions is a fundamental property of complex manifolds.
If the complex manifold can be embedded holomorphically into a complex Euclidean space, then there exist a lot of holomorphic functions.
A far deeper theorem is proved by Siu-Yau \cite{Siu_Yau} and Greene-Wu \cite{Greene_Wu} that a complete simply connected nonpositively curved K\"ahler manifold of faster than quadratic sectional curvature decay is biholomorphic to $\mathbb C^n$ by constructing $n$ holomorphic functions. In \cite{Bland}, Bland presented two sufficient conditions given for a complete K\"ahler manifold of nonpositive sectional curvature to admit nonconstant bounded holomorphic functions.

On the other hand, if the manifold is compact, then there exist no nonconstant holomorphic functions.
In \cite{Greene_Wu}, Greene-Wu showed that any complete simply connected K\"ahler manifold with nonpositive sectional curvature does not admit nonzero $L^p$ $(1\leq p<\infty)$ holomorphic functions. Moreover if the curvature tensor has certain stronger conditions, then it does not have any nonconstant bounded holomorphic function neither.

Let $\Delta:=\{z\in \mathbb C : |z|<1\}$ be the unit disc and $\Gamma\subset \text{Aut}(\Delta)$ be a cocompact torsion-free lattice. Then Ohsawa observed that there exists a holomorphic function induced from the Poincar\'e series given by
$$\sum_{\gamma\in\Gamma} \left( \gamma(z)-\gamma(w)\right)^k$$
with any $k\geq 2$ on the disc bundle $\Delta\times\Delta /\Gamma$ over the Riemann surface $\Delta / \Gamma$ where $\Gamma$ acts on $\Delta\times\Delta$ by $\gamma(z,w)=(\gamma z, \gamma w)$.
In \cite{Adachi} Adachi gave a concrete description of $\mathcal O(\Delta\times\Delta/\Gamma)$, and proved that there exist weighted $L^2$ holomorphic functions on $\Delta\times\Delta / \Gamma$ out of holomorphic sections of $K^\ell$ with $\ell\in \mathbb N$ where $K$ denotes the canonical line bundle over $\Delta / \Gamma$.
In \cite{Lee_Seo} the authors generalized his method to the $\mathbb{B}^n$-fiber bundle $\mathbb B^n\times \mathbb B^n/\Gamma$ over a complex hyperbolic space form $\mathbb B^n / \Gamma$.

Now let $\widetilde M$ be a complex manifold and $\Gamma$ be a torsion-free cocompact lattice of $\text{Aut}(\widetilde M)$ and $\rho\colon\Gamma\to \text{Aut}(\mathbb B^N)$ be a homomorphism. Suppose that there exists a $\rho$-equivariant holomorphic embedding $\imath \colon \widetilde M\to\mathbb B^N$, i.e. for any $\gamma\in\Gamma$, $\imath (\gamma \zeta) = \rho(\gamma) \imath(\zeta)$.
Then the $\mathbb B^N$-fiber bundle $M\times_\rho\mathbb B^N:= \widetilde M\times \mathbb B^N/\Gamma$ over $M:=\widetilde M/\Gamma$ where any $\gamma\in \Gamma$ acts on $\widetilde M\times \mathbb B^N$ by $(\zeta,w)\mapsto (\gamma \zeta, \rho(\gamma) w)$ has a holomorphic function which is induced from the following Poincar\'e series
$$
\sum_{\gamma\in \Gamma} \left( \rho(\gamma)_j(\imath (\zeta)) - \rho(\gamma)_j (w)\right)^k
$$
with $(\zeta,w)\in \widetilde M\times\mathbb B^N$ and  $k\geq N+1$ (cf. Corollary 4.10  in \cite{Lee_Seo}).
In this paper we generalize the method given in \cite{Adachi, Lee_Seo} under the conditions when $\imath\colon \widetilde M\to\mathbb B^N$ is a totally geodesic isometric embedding.
It is worth to emphasize that $\mathbb B^N/\rho(\Gamma)$ does not need to be compact.

Recently Deng-Forn\ae ss \cite{Deng_Fornaess} constructed a hyperconvex disc fiber bundle over a Hopf manifold which does not admit any nonconstant holomorphic function. Here hyperconvex means that there exists a bounded plurisubharmonic exhaustion function. Since $\mathbb B^N\times \mathbb B^N/\Gamma$ is hyperconvex (\cite{Adachi_hyperconvex, Seo}), $M\times_\rho\mathbb B^N$ is also hyperconvex for any $N$.
\bigskip

{\bf Acknowledgement}
The authors would like to thank Adachi Masanori for giving useful comments. He observed that $A_{-1}^2(\Omega) = \bigcap_{\alpha \geq -1} A_\alpha^2(\Omega)$ if $N>n$.
{ The authors are grateful to the referee for careful reading of the paper and valuable suggestions and comments.}
The first author was supported by the Institute for Basic Science (IBS-R032-D1). The second author was partially supported by Basic Science Research Program through the National Research Foundation of Korea (NRF) funded by the Ministry of Education (NRF-2022R1F1A1063038).

\section{Raising operators for submanifolds}
\subsection{Raising operators}
Let $X$ be a K\"ahler manifold of dimension $N$, $g$ be its K\"ahler metric and $G \in C^\infty(X, \Lambda^{1,1}T^*_X)$ be its K\"ahler form.
Let $S^mT_X^*$ be the $m$-th symmetric power of holomorphic cotangent bundle $T^*_X$ of $X$.
Let $Y$ be a K\"ahler manifold of dimension $n$ and $\imath \colon Y \rightarrow X$ be a holomorphic map.
We will denote by $\Lambda^{p,q}T_X^*$ the vector bundle of complex-valued $(p,q)$-forms over $X$.

For any $p=0,1$ and $\tau\colon C^\infty(Y, \imath^*(S^m T_X^*))\to C^\infty(Y, \imath^*(S^mT_X^*)\otimes\imath^*(\Lambda^{p,0}T_X^*)\otimes \Lambda^{0,q}T_Y^*)$, define a map 
$$\mathcal R_\tau^m\colon C^\infty(Y, \imath^*(S^mT_X^*))\to C^\infty(Y, \imath^*(S^{m+p}T_X^*)\otimes \Lambda^{0,q}T_Y^*)$$
by
$$
\mathcal R_\tau^m(u)= \sum \tau_{PQ} (u)a^P\otimes \overline b^Q
$$
where $\tau(u) = \sum_{|P|={p}, |Q|=q} \tau_{PQ}(u)\otimes a^P\otimes \overline b^Q$
for $u\in C^\infty(Y,\imath^* (S^mT_X^*))$.
Here $a=(a_1,\ldots, a_N)$ is a local frame of $\imath^* (T_X^*)$ and $b=(b_1,\ldots, b_n)$ is a local frame of $T_Y^*$.

For example, for an orthonormal frame $(e_1,\ldots, e_N)$ the K\"ahler form $G = \sum e_\ell\otimes \overline e_\ell$ defines a map, which will be also denoted by $G$, 
$$
G\colon C^\infty(Y, \imath^*(S^mT_X^*))\to C^\infty(Y, \imath^*(S^mT_X^*)\otimes \imath^*(\Lambda^{1,0}T_X^*)\otimes \Lambda^{0,1}T_Y^*)
$$
by  $u\mapsto \sum_{\ell=1}^N u\otimes e_\ell\circ \imath\otimes \overline{ \imath^* e_\ell}$.
We remark that we use the notation $e_\ell\circ\imath$ or $e|_{\imath(\zeta)}$, $\zeta\in Y$, for the pull-back section and {$\imath^* e_\ell$} for the pull-back form.
The corresponding raising operator $\mathcal R_G$ is defined by 
$$
\mathcal R_G\colon C^\infty (Y,\imath^*(S^mT^*_X))\to C^\infty(Y, \imath^*(S^{m+1}T^*_X)\otimes \Lambda^{0,1}T^*_Y),
$$
\begin{equation}
    u=\sum_J u_J(e\circ\imath)^J \mapsto \sum_{\ell=1}^N\sum_J u_J (e\circ \imath)^J (e_\ell\circ\imath)\otimes \overline{\imath^*  e_\ell}.
\end{equation}
Since this definition does not depend on the choice of orthonormal frame, it is well defined.
By a similar way, for the Chern connection of $\imath^*(S^mT_X^*)$ and its Chern curvature form, we can define the corresponding raising operators (cf. \cite{Lee_Seo}).

\subsection{{Hodge type identities over $M$}}

Let $M$ be a compact K\"ahler manifold of dimension $n$.
Let $\widetilde M$ be its covering and $\Gamma$ be a subgroup in $\text{Aut}(\widetilde M)$
so that $M$ is biholomorphic to $\widetilde M/ \Gamma$.
Let $\rho\colon \Gamma\rightarrow \text{Aut}(\mathbb B^N)$ be a homomorphism and $\imath \colon \widetilde M \rightarrow \mathbb B^N$ be a $\rho$-equivariant holomorphic embedding, i.e. for any $\gamma\in \Gamma$ and $\zeta \in \widetilde M$, $\rho(\gamma)(\imath(\zeta)) = \imath (\gamma(\zeta))$.
We will assume that $\Sigma:= \mathbb B^N/\rho(\Gamma)$ is a complex manifold and the map $\imath$ induces a holomorphic embedding from $M$ to $\Sigma$. For simplicity, we also denote this map by $\imath$. 

Let 
\begin{equation}\label{Bergman metric}
g_{\mathbb{B}^N} (z) 
    =\sum_{j,k=1}^N \left( \frac{(1-|z|^2)\delta_{kj} + z_{j} \overline z_k} {(1-|z|^2)^2} \right) dz_k\otimes d\overline z_j
\end{equation}
be the normalized Bergman metric of $\mathbb{B}^N$. We assume that the normal bundle ${N=N_\imath:=\imath^*T_{\Sigma}/T_M}$ is holomorphically isomorphic to the orthogonal complement of $T_{M}$ in $\imath^* T_\Sigma$  with respect to the induced metric $g$ from $g_{\mathbb{B}^N}$.
We emphasize that $\Sigma$ does not need to be compact.
Let $h := \imath^* g$ be the pull-back metric of $g$ on $M$.

For any measurable section $\phi$ of $\imath^* (S^{m}T_{\Sigma}^{*}) \otimes \Lambda^{p,q} T_{M}^*$, we define an $L^2$-norm by
\begin{equation}\label{defnorm}
\begin{aligned}
\| \phi \|^2= \int_{M} \langle \phi, \phi \rangle  dV_{M}
\end{aligned}
\end{equation}
where $\langle~ ,~  \rangle$ and $dV_M$ are induced by $g$ and $h=\imath^* g$ on $M$.
In particular, if we express any measurable section $\phi$ of $i^*(S^m T_{\Sigma}^*)$ with respect to a local orthonormal frame $\{e_1,\ldots, e_N\}$ of $T^*_\Sigma$ by $\phi= \sum_{I} f_I (\zeta) {e^I |_{\imath (\zeta)}}$,
we have
$$
\langle \phi, \phi\rangle = \sum_{|I|=m} \frac{I!}{m!} |f_{I}|^2.
$$ 
The formal adjoint of $\bar \partial$ on $L^2 (M, \imath^*(S^m T_{\Sigma}^*) \otimes \Lambda^{p,q} T_{M}^*)$  with respect to the induced metric \eqref{defnorm} will be denoted by $\bar \partial^*_{(p,q), M}$.
For simplicity, we will write $\bar \partial^*_{M}$ instead of $\bar \partial^*_{(p,q),M}$, if no confusion is likely to arise.

For the K\"{a}hler metric $g$ on $\Sigma$, let $g^m$ denote the Hermitian metric on $S^{m}T_{\Sigma}^*$ induced from $g$ and let $\Box_{m,M}^k$ be the complex Laplace operator
\begin{equation*}
\Box_{m,M}^{k} : C^{\infty}(M, \imath^* (S^m T_{\Sigma}^*) \otimes \Lambda^{0,k} T^*_{M} ) \rightarrow C^{\infty}(M, \imath^* (S^m T_{\Sigma}^*) \otimes \Lambda^{0,k} T^*_{M})
\end{equation*}
given by
\begin{equation*}
\Box^{k}_{m,M} = \bar \partial_M \bar \partial^*_{M} + \bar \partial^*_{M} \bar \partial_{M}
\end{equation*}
with respect to the metric $g^m$ where $k=0,1$. We will omit $k,m$ in the notation if there is no ambiguity. Especially we simply write $\Box_{M}$ instead of $\Box^{0}_M$. Let $G^1$ be the Green operator of $\Box^{1}_{M}$.

Since the short exact sequence $0\to T_M\to \imath^*T_\Sigma\to {N} \to 0$ holomorphically splits, we have $$
\imath^*(S^m T^*_{\Sigma}) \cong \bigoplus_{\ell=0}^m
{S^\ell T^*_{M} \otimes S^{m-\ell}N^*},
$$
and as a consequence
$$
C^\infty(M,\imath^*(S^m T^*_{\Sigma})) \cong \bigoplus_{\ell=0}^m
C^\infty(M, {S^\ell T^*_{M} \otimes S^{m-\ell}N^*}),
$$
\begin{equation}\label{decomposition}
H^0(M,\imath^*(S^m T^*_{\Sigma})) \cong \bigoplus_{\ell=0}^m
H^0 (M, {S^\ell T^*_{M} \otimes S^{m-\ell}N^* }).
\end{equation}
\begin{lemma}\label{RG formula}
Let $m$,$\ell$ be non-negative integers with $\ell \leq m$. The raising operator $\mathcal{R}_{G}$ is a linear injective map and for any $u \in C^{\infty}(M, {S^\ell T_{M}^* \otimes S^{m-\ell} N^*})$,
\begin{equation}\label{RG1}
\|\mathcal{R}_G (u) \|^2 = \frac{\ell + n }{m+1} \|u \|^2	
\end{equation}
and
\begin{equation}\label{RG2}
\{ \Box_M, \mathcal{R}_G \} (u) := \Box_M\mathcal R_G u - \mathcal R_G\Box_Mu
=   ( m + \ell)  \mathcal R_G(u).	
\end{equation}
\end{lemma}

\begin{proof}
Let $\{e_1,\ldots, e_N\}$ be a local orthonormal frame  of $T^*_\Sigma$ on a small open set $U$ of $\Sigma$
so that $g= \sum_{\tau=1}^N e_{{\tau}} \otimes \bar e_{{\tau}}$ and
{$h=\sum_{\tau=1}^n \imath^*e_{{\tau}} \otimes \imath^*\bar e_{{\tau}}$}.
Let $u= \sum_{|I|=m} {u}_I e^{I}$ be a smooth section of $\imath^*(S^m T_{\Sigma}^*)$. Since $\imath^{*} e_{{\tau}} =0$ for any $n+1 \leq {\tau} \leq N$, we obtain
$$
\mathcal{R}_{G} u = \sum_{{\tau}=1}^{n} \sum_{|I|=m} {u}_I e^{I} e_{{\tau}} \otimes \overline{\imath^{*} e_{{\tau}}}.
$$
If $ u \in C^{\infty}(M, {S^\ell T_{M}^* \otimes S^{m-\ell}N^*})$, then one has
\begin{equation*}
\begin{aligned}
\langle \mathcal{R}_{G} u , \mathcal{R}_{G} u \rangle &= \sum_{{\tau},m=1}^{n} \langle ue_{{\tau}} \otimes \overline{\imath^{*} e_{{\tau}}} , u e_m \otimes \overline{\imath^{*} e_m} \rangle
=\sum_{{\tau}=1}^{n} \langle u e_{{\tau}} , u e_{{\tau}} \rangle  \\
&= \sum_{|I|=m} \sum_{{\tau}=1}^{n} \frac{i_1! \cdots (i_{{\tau}} +1)! \cdots i_n! \cdots i_N!}{(m+1)!} | {u}_I |^2\\
&= \sum_{|I|=m} \frac{\big( \sum_{j=1}^{n} i_j + n  \big)}{m+1}  \frac{I !}{m!}  |{u}_I |^2 \\
&{ = \sum_{|I|=m} \frac{\ell  + n }{m+1} \frac{I !}{m!}  |{u}_I |^2 
= \frac{\ell+n}{m+1}\|u\|^2.}
\end{aligned}
\end{equation*}
which implies \eqref{RG1}.

To prove \eqref{RG2}, let $q\in M$ and $(\zeta_1, \ldots, \zeta_n)$ be a holomorphic normal coordinate system on a small open set $q \in U \subset M$ such that $U \cong \imath(U) \subset \Sigma$. Let $p:=\imath (q) \in \imath(U)$ and take a holomorphic normal coordinate system $(z_1,\ldots, z_N)$ at $p$ such that for $\imath_k:=z_k\circ \imath$, $k=1,\ldots, N$
we have
\begin{equation}\label{totally geo}
\frac{\partial \imath_m}{\partial \zeta_\mu} \bigg|_{q} =\begin{cases}
\delta_{m \mu} & (m=1,\ldots,n),\\
0 & (m=n+1,\ldots,N).
\end{cases}
\end{equation}

Let $\{ h_{\gamma} \}$ be a holomorphic normal frame of $\imath^* \left(S^m T_{\Sigma}^*\right)$. 
For any smooth section $u = \sum_{\gamma} u_{\gamma} h_{\gamma}$ of $\imath^*(S^m T_{\Sigma}^*$), we have
\begin{equation*}
\begin{aligned}
\bar \partial^*_M \mathcal{R}_{G}(u) &= \bar \partial^*_M \left( \sum_{\alpha, \beta=1}^{N} \sum_{\gamma} u_\gamma (g_{\alpha \beta} \circ \imath)h_{\gamma} (dz_{\alpha} \circ \imath) \otimes \overline{\imath^* d z_{\beta}} \right) \\
&= \bar \partial^*_M \left( \sum_{\alpha, \beta=1}^{N}  \sum_{{\tau}=1}^{n} \sum_{\gamma} u_\gamma (g_{\alpha \beta} \circ \imath) h_{\gamma}  (dz_{\alpha} \circ \imath) \otimes \overline{\frac{\partial \imath_{\beta}}{\partial \zeta_{{\tau}}}}  d \bar \zeta_{{\tau}} \right) \\
& = -\sum_{\alpha, \beta =1}^{N} \sum_{k=1}^n \sum_{\gamma} \frac{\partial \big(u_\gamma (g_{\alpha \beta} \circ \imath) \overline{\frac{\partial \imath_{\beta}  }{\partial \zeta_k }  } \big) }{\partial \zeta_k} h_{\gamma}(dz_{\alpha} \circ \imath)  + O(|\zeta|) 
\end{aligned}
\end{equation*}
\begin{equation*}
\begin{aligned}
\qquad \qquad \qquad \qquad \qquad &=  - \sum_{\alpha, \beta =1}^{N} \sum_{k=1}^n \sum_{\gamma} \left( \overline{\frac{\partial \imath_{\beta}}{\partial \zeta_{k}}} \frac{\partial (g_{\alpha \beta} \circ \imath) }{\partial \zeta_k} u_{\gamma}
+\overline{\frac{\partial \imath_{\beta}}{\partial \zeta_{k}}} \frac{\partial  u_{\gamma}}{\partial \zeta_k}(g_{\alpha \beta} \circ \imath)\right)h_\gamma  (dz_{\alpha} \circ \imath) + O(|\zeta|)\\
&= - \sum_{k, \alpha=1}^{n} \sum_{\gamma} \overline{\frac{\partial \imath_{\alpha}}{\partial \zeta_{k}}} \frac{\partial  u_{\gamma}}{\partial \zeta_k} h_\gamma (dz_{\alpha} \circ \imath)  + O(|\zeta|).
\end{aligned}
\end{equation*}
By \eqref{totally geo}, we obtain
$$
\bar \partial^*_M \mathcal{R}_G (u) = - \sum_{k=1}^{n} \sum_{\gamma} \frac{\partial  u_{\gamma}}{\partial \zeta_k}h_\gamma (dz_{k} \circ \imath)
$$
at the point ${q}$.
Let $\widetilde D_m$ be {the $(1,0)$ part} of the Chern connection of $\imath^*\left(S^m T_{\Sigma}^*\right)$ induced from $(S^m T_{\Sigma}^*, g^m)$. 
Then 
\begin{equation*}
\begin{aligned}
{\widetilde D_m (u)} = 
\widetilde D_m \left( \sum_{\gamma} {u}_\gamma h_\gamma \right) &= \sum_{k=1}^{n} \sum_{\gamma} \frac{\partial {u}_\gamma}{\partial \zeta_k} h_\gamma \otimes d \zeta_k + \sum_{k=1}^{n} \sum_{\gamma, \mu} {u}_{\gamma} {\theta}_{\gamma k}^{\mu} h_{\mu}  \otimes d \zeta_{k}
\end{aligned}
\end{equation*}
where ${\theta}_{\gamma k}^{\mu}$ is the connection one form of $\widetilde D_m$. 
Then it follows that
\begin{equation*}
\begin{aligned}
\mathcal{R}_{\widetilde D_m} \left( \sum_{\gamma} u_{\gamma} h_{\gamma} \right) 
&= \sum_{k=1}^{n} \sum_{\gamma} \frac{\partial u_{\gamma} }{\partial \zeta_k} h_{\gamma} (dz_{k} \circ \imath) + \sum_{k=1}^{n} \sum_{\gamma, \mu}  u_\gamma  \theta_{\gamma k}^{\mu}h_\mu (dz_{k} \circ \imath).
\end{aligned}
\end{equation*}
Hence
$$
\bar \partial^*_M \mathcal{R}_{G} u = - \mathcal{R}_{\widetilde D_m} u
$$
at ${q}$. 

Let $\theta_{\gamma}^{\mu} = \sum_k \theta_{\gamma k}^{\mu} d \zeta_k$.
Since the Chern curvature form $\Theta (\imath^* (S^m T_{\Sigma}^*))$ of $\imath^*(S^m T_{\Sigma}^*)$  satisfies $\Theta\left(\imath^*\left( S^m T_{\Sigma}^* \right) \right)= \bar \partial \theta$ for $\theta:=(\theta_\gamma^{\mu})$, it follows that
$$
\frac{\partial {\theta}^{\mu}_{\gamma k}}{\partial \bar \zeta_\lambda} = -R^{\mu}_{\gamma k \bar \lambda}
$$
where $R^{\mu}_{\gamma k \bar \lambda}$ is given by $$
\Theta\left(\imath^*\left( S^m T_\Sigma^* \right)\right)= \sum_{\alpha, \beta} \sum_{s,t=1}^{n} R^{\beta}_{\alpha s \bar t } h_{\alpha}^* \otimes h_{\beta} \otimes d\zeta_s \wedge d\bar \zeta_t.$$ 
Therefore by using \eqref{totally geo}, it follows that
\begin{equation}\label{curvature1}
\begin{aligned}
&\bar \partial_M \bar \partial^*_M \mathcal{R}_{G} u = - \bar \partial_M \mathcal{R}_{\widetilde D_m} u \\
&= - \sum_{k, \lambda=1}^{n} \sum_{\gamma}  \frac{\partial^2 u_{\gamma}}{\partial  \zeta_k \partial\bar \zeta_{\lambda}} h_\gamma (dz_{k} \circ \imath) \otimes d  \bar\zeta_{\lambda} 
+ \sum_{k,m=1}^{n} \sum_{\gamma, \mu}u_\gamma R^{\mu}_{\gamma k \bar m}h_{\mu} (dz_{k} \circ \imath) \otimes {d\overline\zeta_m} + O(|\zeta|).
\end{aligned}
\end{equation}
Furthermore, \small
\begin{equation}\label{calculation to Box}
\begin{aligned}
\bar \partial^*_M \bar \partial_M \mathcal{R}_{G} u 
&= \sum_{\alpha, \beta =1}^{N} \sum_{\gamma} \bar \partial^*_M \bar \partial_M \bigg( u_\gamma (g_{\alpha \beta} \circ \imath)  h_{\gamma}  (dz_{\alpha} \circ \imath) \otimes \overline{\imath^* d z_{\beta}} \bigg) \\
&= \sum_{\alpha,\beta=1}^{N} \sum_{\lambda=1}^{n} \sum_{\gamma} \bar \partial^*_M \bar \partial_M \bigg( u_\gamma (g_{\alpha \beta} \circ \imath) h_{\gamma}  (dz_{\alpha} \circ \imath) \otimes \overline{\frac{\partial \imath_{\beta}}{\partial \zeta_\lambda}}d \bar{ \zeta _{\lambda}}   \bigg) \\
&=  \bar \partial^*_M \left( \sum_{\alpha,\beta=1}^{N} \sum_{\lambda,  k=1}^{n} \sum_{\gamma}  \frac{\partial \big( u_\gamma (g_{\alpha \beta} \circ \imath) \overline{\frac{\partial \imath_\beta}{\partial \zeta_\lambda}} \big)}{\partial \bar \zeta_k}  h_{\gamma}
(dz_{\alpha} \circ \imath) \otimes d\bar\zeta_k \wedge d \bar \zeta_\lambda  \right) \\
&= -  \sum_{\alpha,\beta=1}^{N} \sum_{\lambda,  k=1}^{n} \sum_{\gamma}  \bigg( \frac{\partial^2 \big( u_\gamma (g_{\alpha \beta} \circ \imath) \overline{\frac{\partial \imath_\beta}{\partial \zeta_\lambda}} \big)}{\partial \zeta_k \partial \bar \zeta_k}   d\bar\zeta_\lambda - \frac{\partial^2 \big( u_\gamma (g_{\alpha \beta} \circ \imath) \overline{\frac{\partial \imath_\beta}{\partial \zeta_\lambda}} \big)}{\partial \zeta_\lambda \partial \bar \zeta_k}  d \bar \zeta_k  \bigg) h_{\gamma}
(dz_{\alpha} \circ \imath)  +O(| \zeta |) \\
&= -\sum_{\alpha =1}^{N}\sum_{\beta,k=1}^{n} \sum_\gamma \left( \frac{\partial^2 \big( u_\gamma (g_{\alpha \beta} \circ \imath)  \big) }{\partial \zeta_k \partial \bar \zeta_k} d \bar \zeta_{\beta} -\frac{\partial^2 \big( u_\gamma (g_{\alpha \beta} \circ \imath) \big) }{\partial \zeta_\beta \partial \bar \zeta_k}  d \bar \zeta_{k}\right)h_{\gamma} (dz_{\alpha} \circ \imath) \\
&\quad - \sum_{\alpha, \beta=1}^{N}\sum_{\lambda, k=1}^n \sum_\gamma \bigg( \frac{\partial (u_\gamma (g_{\alpha\beta} \circ \imath))}{\partial \zeta_k}  d \bar \zeta_\lambda - \frac{\partial (u_\gamma (g_{\alpha \beta} \circ \imath))}{\partial \zeta_\lambda}  d \bar \zeta_k \bigg)\overline{\frac{\partial^2 \imath_\beta}{\partial \zeta_k \partial \zeta_\lambda}} h_\gamma (dz_\alpha \circ \imath)
+ O(|\zeta|)\\
&=-\sum_{\alpha =1}^{N}\sum_{\beta,k=1}^{n} \sum_\gamma \left( \frac{\partial^2 \big( u_\gamma (g_{\alpha \beta} \circ \imath)  \big) }{\partial \zeta_k \partial \bar \zeta_k} d \bar \zeta_{\beta} -\frac{\partial^2 \big( u_\gamma (g_{\alpha \beta} \circ \imath) \big) }{\partial \zeta_\beta \partial \bar \zeta_k}  d \bar \zeta_{k}\right)h_{\gamma} (dz_{\alpha} \circ \imath) + O(|\zeta|). 
\end{aligned}
\end{equation}
Note that the last equality of \eqref{calculation to Box} follows by
\begin{equation*}
\begin{aligned}
-& \sum_{\alpha, \beta=1}^{N}\sum_{\lambda, k=1}^n \sum_\gamma \bigg( \frac{\partial (u_\gamma (g_{\alpha\beta} \circ \imath))}{\partial \zeta_k}  d \bar \zeta_\lambda - \frac{\partial (u_\gamma (g_{\alpha \beta} \circ \imath))}{\partial \zeta_\lambda}  d \bar \zeta_k \bigg)\overline{\frac{\partial^2 \imath_\beta}{\partial \zeta_k \partial \zeta_\lambda}} h_\gamma (dz_\alpha \circ \imath)\\
&=-\sum_{\alpha, \beta=1}^{N}\sum_{\lambda, k=1}^n \sum_\gamma \bigg( \frac{\partial (u_\gamma (g_{\alpha\beta} \circ \imath))}{\partial \zeta_\lambda}  d \bar \zeta_k \bigg) \overline{\frac{\partial^2 \imath_\beta}{\partial \zeta_k \partial \zeta_\lambda}} h_\gamma (dz_\alpha \circ \imath)\\
&\quad - \sum_{\alpha, \beta=1}^{N}\sum_{\lambda, k=1}^n \sum_\gamma \bigg( \frac{\partial (u_\gamma (g_{\alpha \beta} \circ \imath))}{\partial \zeta_\lambda}  d \bar \zeta_k \bigg)\overline{\frac{\partial^2 \imath_\beta}{\partial \zeta_k \partial \zeta_\lambda}} h_\gamma (dz_\alpha \circ \imath)
\\
&= 0.
\end{aligned}
\end{equation*}
Moreover,
\begin{equation*}
\begin{aligned}
&\frac{\partial^2 \big( u_\gamma (g_{\alpha \beta} \circ \imath) \big) }{\partial \zeta_k \partial \bar \zeta_k} d \bar \zeta_\beta - \frac{\partial^2 \big( u_\gamma (g_{\alpha \beta} \circ \imath)  \big) }{\partial \zeta_\beta \partial \bar \zeta_k} d \bar \zeta_k \\
&= \left( \frac{\partial^2 (g_{\alpha \beta} \circ \imath)}{\partial \zeta_k \partial \bar \zeta_k} d \bar \zeta_\beta - \frac{\partial^2 (g_{\alpha \beta} \circ \imath )}{\partial \zeta_{\beta} \partial \bar \zeta_k} d \bar \zeta_k \right)u_{\gamma }
+(g_{\alpha \beta} \circ \imath) \left( \frac{\partial^2 u_{\gamma}}{\partial \zeta_k \partial \bar \zeta_k} d \bar \zeta_\beta- \frac{\partial^2 u_{\gamma}}{\partial \zeta_{\beta} \partial \bar \zeta_k} d \bar \zeta_k \right)+O(|\zeta|).
\end{aligned}
\end{equation*}
Since $g$ is K\"ahler, it follows that
\begin{equation}\label{comm1}
\begin{aligned}
\bar \partial^*_M \bar \partial_M \mathcal{R}_{G} u 
&= -\sum_{\alpha =1}^{N} \sum_{\beta, k=1}^{n} \sum_{\gamma}(g_{\alpha \beta} \circ \imath) h_{\gamma} (dz_{\alpha} \circ \imath) \otimes \left( \frac{ \partial^2 u_\gamma  }{\partial \zeta_k \partial \bar \zeta_k } d \bar \zeta_\beta - \frac{\partial^2 u_{\gamma}}{\partial \zeta_\beta \partial \bar \zeta_k} d \bar \zeta_k  \right) + O(|\zeta|) .
\end{aligned}
\end{equation}
Since
\begin{equation}\label{comm2}
\begin{aligned}
\mathcal{R}_{G} \bar \partial^*_{M} \bar \partial_{M} u &= - \sum_{\alpha,\beta=1}^N\sum_{\mu=1}^n \sum_{\gamma}  \frac{\partial^2 u_{\gamma} }{\partial \zeta_{\mu} \partial \bar \zeta_{\mu}} h_{\gamma} (g_{\alpha \beta} \circ \imath) (dz_{\alpha} \circ \imath) \otimes \overline{\imath^* d z_{\beta}} + O(|\zeta|),
\end{aligned}
\end{equation}
by adding \eqref{curvature1}, \eqref{comm1}, \eqref{comm2}, and using \eqref{totally geo} we have
\begin{equation*}
\begin{aligned}
\{ \Box, \mathcal R_{G} \} u &= \bar\partial_M \bar \partial^*_M \mathcal{R}_{G} u + \bar \partial^*_M \bar \partial_M \mathcal{R}_{G} u - \mathcal{R}_{G} {\bar \partial_{{M}}^{*} \bar \partial_{{M}}} u\\ 
&=  \sum_{k,j=1}^{n} \sum_{\gamma, \mu} u_\gamma   R^{\mu}_{\gamma k \bar j} d\zeta_k h_{\mu}  \otimes d \bar \zeta_j.
\end{aligned}
\end{equation*}
Moreover at ${q}$, we have
\begin{equation*}
\sum_{k,m=1}^{n} \sum_{\gamma, \mu} u_\gamma  R^{\mu}_{\gamma k \bar m} d{\zeta}_k h_{\mu}  \otimes d \bar \zeta_m  =  \sum_{|I|=m} \sum_{j=1}^{N} i_j u_{I} e_1^{i_1} \cdots e_j^{i_j-1} \cdots e_N^{i_N} \cdot \mathcal{R}_{\Theta(\imath^* T_{\Sigma}^* )}(e_j)
\end{equation*}
where ${\mathcal{R}_{\Theta(\imath^* T_{\Sigma}^*)}}$ is defined by
$$
\mathcal{R}_{\Theta(\imath^* T_{\Sigma}^*)} (e_j) = \sum_{a=1}^{N}\sum_{k,m=1}^{n}e_a \otimes \Theta(T_{\Sigma}^*)^{a}_{jkm} e_k \wedge \bar e_m.
$$
From
$$
\Theta(T_{\Sigma})^{a}_{jkm}=-(\delta_{ak}\delta_{jm}+\delta_{aj}\delta_{km}),
$$
one has
$$
\Theta(T_{\Sigma}^*)^{a}_{jkm}=-\Theta(T_{\Sigma})^*= \delta_{jk}\delta_{am}+\delta_{ja}\delta_{mk}
$$
and hence
\begin{equation*}
\begin{aligned}
{\mathcal{R}_{\Theta(\imath^* T_{\Sigma}^*)}}(e_j)= \sum_{a=1}^{N}\sum_{k,m=1}^{n} e_a \otimes\Theta(T_{\Sigma}^*)^a_{jkm} e_k \wedge \bar e_m  = \sum_{a=1}^{n}  {\epsilon_j e_a \otimes e_j \wedge \bar e_a} + e_j \otimes \sum_{r=1}^{n}e_r \wedge \bar e_r
\end{aligned}
\end{equation*}
where 
\begin{equation}
\epsilon_\mu=
\begin{cases}
1 & \text{ if } \mu \in \{1, \cdots, n \}, \\
0 & \text{ otherwise. }
\end{cases}
\end{equation}
Therefore we have 
\begin{equation*}
\begin{aligned}
&\sum_{|I|=m}\sum_{j=1}^{N}i_j {u_I} e_1^{i_1} \cdots e^{i_j-1}_{j} \cdots e_N^{i_N} \cdot \mathcal{R}_{\Theta(\imath^* T_{\Sigma}^*)}(e_j)  \\
&= \sum_{|I|=m} \sum_{j=1}^{N}i_j {u_I} e_1^{i_1} \cdots e_j^{i_j-1} \cdots e_N^{i_N} 
\cdot \left( \sum_{a=1}^n e_a \otimes e_j \wedge \bar e_a + e_j \otimes \sum_{r=1}^n e_r \wedge \bar e_r \right)  \\
&=\sum_{|I|=m}\sum_{r=1}^{n}\left(m+\sum_{j=1}^{n} i_j \right) {u_I} e_1^{i_1} \cdots e_n^{i_n} \cdots e_{N}^{i_N} \cdot e_r \otimes \bar e_r \\
&= \sum_{|I|=m} \left( m+ \sum_{j=1}^{n} i_j \right) \mathcal{R}_{G} (u_I e^I).
\end{aligned}
\end{equation*}
and it implies \eqref{RG2} if $ u \in C^{\infty}(M, {S^\ell T_{ M}^* \otimes S^{m-\ell}N^*})$.
\end{proof}

\begin{remark}
\begin{enumerate}
\item  {If  $\imath$ is totally geodesic,}
since up to the composition with an automorphism of $\mathbb B^N$ we have
 $$\imath (\widetilde M)=\mathbb B^N\cap \{(z_1,\ldots, z_n,0,\cdots,0)\in \mathbb C^N : z_j\in \mathbb C, \forall j\}$$ 
and the normalized Bergman metric of $\mathbb B^N$ is given by 
\eqref{Bergman metric}, 
 it follows that the normal bundle $\imath^*T_{\mathbb B^N}/ T_{\widetilde M}$ is holomorphically isomorphic to the orthogonal complement of $T_{\widetilde M}$ in $\imath^*T_{\mathbb B^N}$.
This implies that under the condition given in Theorem~\ref{main}, the normal bundle $\imath^*T_{\Sigma}/T_M$ is holomorphically isomorphic to the orthogonal complement of $T_{M}$ in $\imath^* T_\Sigma$  with respect to the induced metric $g$ from the Bergman metric of $\mathbb B^N$.

\item
For a compact manifold $M$ and a holomorphic embedding, not necessarily totally geodesic, $\imath\colon M\rightarrow {\Sigma:={\mathbb B}^N/\Gamma}$,
let $\{e_1, \cdots, e_N \}$ be a local orthonormal frame of $T_{\Sigma}^*$ so that $g= \sum_{\ell=1}^{N} e_{\ell} \otimes \bar e_{\ell}$ and {$h:=\imath^*g=\sum_{\ell=1}^{n} i^* e_\ell \otimes i^* \bar e_\ell$.} For any smooth section 
$$
u= \sum_{i_1+\cdots +i_N=m} {u}_{i_1 \cdots i_N} e_1^{i_1} \cdots e_n^{i_n} \cdots e_N^{i_N}
$$
of $\imath^* (S^m T_{\Sigma}^*)$, we have
\begin{equation*}
\|\mathcal{R}_G (u) \|^2 = \sum_{|I|=m}\frac{ i_1 +\cdots + i_n  + n }{m+1} \|u_I e^I \|^2	
\end{equation*}
and
\begin{equation*}
\{ \Box_M, \mathcal{R}_G \} (u) := \Box_M\mathcal R_G u - \mathcal R_G\Box_Mu
=  \sum_{|I|=m} ( m + i_1 + \cdots + i_n)  \mathcal R_G(u_I e^I).	
\end{equation*}
\end{enumerate}
\end{remark}

For each positive integer $m$, non-negative integer $\ell $ with $\ell\leq m$ and $k=0,1$, let $\Box_{m,M}^{k,\ell}$ denote the complex Laplace operator on 
\begin{equation*}
{S^\ell T_{M}^* \otimes S^{m-\ell} N^*}\otimes \Lambda^{0,k}T^*_{M}\subset \imath^*(S^m T^*_{\Sigma})\otimes \Lambda^{0,k}T^*_M
\end{equation*}
over $M$.

\begin{corollary}\label{eigenvalue}
Let $\ker^{\perp} (\Box_{m,M}^{0,\ell} - \lambda I)$ be the orthogonal complement of $\ker(\Box_{m,M}^{0,\ell} - \lambda I)$ in $L^2(M, {S^\ell T_{M}^* \otimes S^{m-\ell} N^*}   )$. Then one has
\begin{enumerate}
\item $\mathcal{R}_{G} \big( \ker(\Box_{m,M}^{0,\ell} - \lambda I) \big) \subset \ker \big(\Box_{m+1,M}^{1,\ell} - (\lambda+{m+\ell}) I \big)$,
\item $\mathcal{R}_{G} \big( \ker^{\perp} (\Box_{m, M}^{0,\ell} - \lambda I ) \big) \subset \ker^{\perp} \big(\Box_{m+1,M}^{1,\ell} - (\lambda+{m+\ell}) I \big)$.
\end{enumerate}
\end{corollary}
\begin{proof}
(1) is a consequence of the equation \eqref{RG2}.
In view of \cite[Corollary~3.16]{Demailly}, 
 $\ker^{\perp}(\Box_{m,M}^{0,\ell}-\lambda I)$ is the direct sum of $\ker \Box_{m,M}^{0,\ell}$ and  eigenspaces of $\Box_{m, M}^{0,\ell}$ whose eigenvalues are different from $\lambda$. Therefore, (2) follows by Lemma~\ref{RG formula} and self-adjointness of $\Box^{1,\ell}_{m+1,M}$.
\end{proof}

\section{Construction of holomorphic functions on $M \times_\rho \mathbb B^N$}

\subsection{Preliminaries}

Let $z$ be a fixed point in the unit ball $\mathbb{B}^N$. For one dimensional vector space $[z]$ spanned by $z$, we define an orthogonal projection $P_{z}$ from $\mathbb{C}^N$ onto $[z]$. Another orthogonal projection $Q_{z}$ is defined by $P_{z}+Q_{z} = Id_{z}$. Consider an automorphism $T_z$ of $\mathbb{B}^N$ given by
$$
T_z (w) = \frac{z-P_z (w) - s_z Q_z (w)}{1- w \cdot \bar z}
$$
where $s_z = \sqrt{1-|z|^2}$ with ${|z|^2= z \cdot \bar z}$. 
We remark that $T_z$ is an involution, i.e. $T_z \circ T_z = \text{Id}_{\mathbb{B}^N}$.

Let $A=(A_{jk}) := dT_{z} (z)$ and let
\begin{equation}\label{e}
	e_j := \sum_{k=1}^{N} A_{jk}dz_k.
\end{equation}
Then
$\{ e_j \}_{j=1}^{N}$ is an orthonormal frame of $T_{\mathbb B^N}^{*}$ with respect to the Bergman metric on $\mathbb{B}^N$ (see \cite{Lee_Seo}). Let $\{ X_j\}_{j=1}^N$ be the dual frame of $\{ e_j \}_{j=1}^{N}$ on $T_{\mathbb{B}^N}$
i.e.
$$
X_j = \sum_{k=1}^NA^{kj} \frac{\partial}{\partial z_k}
$$
where $( A^{kj})_{j,k=1}^N$ is the inverse matrix of $\left(A_{jk}\right)_{j,k=1}^N$.

Let $\widetilde e_1, \ldots, \widetilde e_n$ be a local orthonormal frame  on $T^*_M$. Then there exist locally defined smooth functions $b_{kl}$ such that
$$
\imath^{*} e_k = \sum_{l=1}^{n} b_{kl} \widetilde e_l.$$
Let $Y_1, \ldots, Y_n$ be the local dual frame of $\widetilde e_1, \ldots, \widetilde e_n$ on $T_M$. Then there exist locally defined smooth functions $y_{lj}$ on $M$ and $a_{lk}$ on $\imath(M)$ such that
\begin{equation}\label{YY}
Y_j = \sum_{l=1}^n {y_{lj}} (\zeta) \frac{\partial}{\partial \zeta_l}
\end{equation}
and
\begin{equation}\label{YX}
\imath_{*} Y_k = \sum_{l=1}^{N} ({a_{lk}} \circ \imath)(\zeta) X_l.
\end{equation}

\begin{lemma}\label{ab}
$b_{km} = a_{km} \circ \imath$
\end{lemma}
\begin{proof}
Since one has 
$$
\imath^* e_k (Y_m) = \sum_l b_{kl} \widetilde e_l (Y_m)= b_{km}
$$ 
and 
$$
\imath^* e_k (Y_m) = e_k (\imath_{*} Y_m) = e_k \left( \sum a_{lm} \circ \imath  X_l \right) = a_{km}\circ \imath ,
$$
we obtain the lemma.
\end{proof}

\begin{lemma}\label{useful formula}
For each $\mu =1,\ldots, N,$
$$
\sum_{l=1}^{N} ( a_{lk} \circ \imath )(\zeta) A^{\mu l} = \sum_{l=1}^{n}  y_{lk}(\zeta)  \frac{\partial \imath_{\mu}}{\partial \zeta_l}.
$$
\end{lemma}
\begin{proof}
Since we have
$$
\imath_* Y_k = \imath_* \left(\sum_{l=1}^n y_{lk} \frac{\partial}{\partial \zeta_l}\right)
= \sum_{l=1}^n \sum_{\mu=1}^N y_{lk} (\zeta) \frac{\partial \imath_{\mu}}{\partial \zeta_l} \frac{\partial}{\partial z_{\mu}} \bigg |_{z=\imath(\zeta)}
$$
and
$$
\sum_{l=1}^{N} (a_{lk} \circ \imath) (\zeta) X_l = \sum_{l=1}^{N} \sum_{\mu=1}^{N} (a_{lk} \circ \imath)(\zeta) A^{\mu l} \frac{\partial}{\partial z_{\mu}} \bigg |_{z=\imath(\zeta)},
$$
by \eqref{YX} the proof is completed.
\end{proof}

\subsection{Definition of formal series}

First, we note that $T_{\Sigma}^* $ is Griffiths positive. Since $\imath$ is an { embedding}, $\imath^* T_{\Sigma}^*$ is also Griffiths positive and so it is ample. Since $\imath^* (S^m T_{\Sigma}^*) \cong S^m (\imath^* T_{\Sigma}^*)$, we know that $\bigoplus_{m=0}^\infty H^0(M, \imath^* (S^m T_{\Sigma}^*) ) \cong \bigoplus_{m=0}^\infty H^0(M, S^m (\imath^* T_{\Sigma}^*) )  $ is infinite dimensional.

By the decomposition \eqref{decomposition} any symmetric differential $\psi \in H^0 (M, \imath^*(S^m T_\Sigma ^*) ) $ is of the form
$$
\psi= \sum_{\ell=0}^{m} \psi^{\ell}_{{m}}
$$
where $\psi^{\ell}_{{m}} \in H^0(M, {S^\ell T_{M}^* \otimes S^{m-\ell} N^*} )$.
Fix $\psi =\sum_{\ell=0}^{m_0} \psi_{m_0}^\ell \in H^0(M, \imath^* (S^{m_0} T^*_\Sigma ) )$. {For each $\ell=0,\ldots, m_0$, we define a sequence of vector {bundles} $\{ F_k^{\ell} \}_{k=0}^{\infty}$ by 
\begin{equation*}
F_{k}^{\ell}=
\begin{cases}
\imath^* (S^k T_{\Sigma}^*) & \text{ if } k < m_0, \\
S^{\ell+k-m_0} T_{M}^* \otimes S^{m_0-\ell} N^* & \text{ if } k \geq m_0.
\end{cases}
\end{equation*}
and consider {the sequence}

$$
\{\varphi^{\ell}_{k}\}_{k=0}^\infty \in \bigoplus_{k=0}^\infty
C^{\infty}(M, F_k^{\ell})$$
such that
\begin{equation}\label{cases}
\varphi^{\ell}_{k}=
\begin{cases}
0 & \text{ if } k < m_0, \\
\psi_{m_0}^{\ell} & \text{ if } k=m_0,\\
\text{the minimal solution of } \\\bar\partial_M \varphi_{k}^{\ell} = -(k-1) \mathcal R_G \varphi_{k-1}^{\ell}
& \text{ if } k>m_0.
\end{cases}
\end{equation}
The minimal solution of the equation
\begin{equation}\label{equation}
\bar \partial_{M} \varphi_{k}^{\ell} = - (k-1) \mathcal R_{G} \varphi_{k-1}^{\ell}
\end{equation}
exists by the following lemma for each $k$.
\begin{lemma}\label{existence}
For any  symmetric differential $\psi =\sum_{\ell=0}^{m_0} \psi_{m_0}^\ell\in H^{0}(M, \imath^* (S^{m_0} T_\Sigma^*) )$ and each $\ell =0,\ldots, m_0$, the sequence $\{\varphi_{k}^{\ell} \}_{k=0}^{\infty}$ given by \eqref{cases} is well defined and it satisfies
$$
\| \varphi_{m_0+m}^\ell \|^2 =  \bigg( \prod_{j=1}^{m} \left( \frac{(\ell+j)+(n-1)}{m_0+j}   \right) \bigg) \left( \frac{(m_0+\ell-1)!}{ \{ (m_0-1)!  \}^2}
\frac{ \{ (m_0+m-1)! \}^2} { (m_0+\ell+m-1)!}\frac{1}{m! } \right) \| \psi^{\ell}_{m_0} \|^2
$$
for any $m \geq 1$. Moreover for any $m \geq 0$, $\varphi_{m_0+m}^\ell$ satisfies
$$
\Box^{0}_{m_0+m,M} (\varphi_{m_0+m}^\ell) = (m^2 + (m_0+{\ell}-1)m) \varphi_{m_0+m}^\ell.
$$
\end{lemma}

\begin{proof}
We will use induction with respect to the index $k$.   If $k\leq m_0$,    then  \eqref{cases}  holds   trivially. 
Suppose that there is the minimal solution of \eqref{equation} for any $k \leq m_0+m-1$. First we will show that $\mathcal R_G(\varphi^\ell_{m_0 + m-1})$ is $\bar\partial_M$-closed. 
{Take a point $q \in M$ and small open set $q \in U \subset {M}$ such that $U \cong \imath (U) \subset \Sigma$. 
Let $(z_1,\cdots, z_N)$ be a local coordinate system at $p :=\imath(q) \in \imath(U)$ such that 
$$
\imath(U) = \{(z_1,\cdots, z_N): z_{n+1}=\cdots = z_N=0 \} \quad \text{near} \quad p=(0,\cdots,0).
$$
Then under the identification $U \cong \imath(U)$, $\{dz^L \}$ becomes a holomorphic frame on $(S^{\ell+m-1} T_{M}^{*}) |_{U} $ and $\{dz^J \}$ becomes a holomorphic frame on $(S^{m_0-\ell} N^*)|_{U}$, where $L:=(i_1,\ldots,i_n)$ and $J:=(i_{n+1},\ldots,i_N)$.}
Write
$$
\varphi_{m_0+m-1}^\ell = \sum_{\substack{|L|=\ell+m-1 \\ |J|=m_0-\ell}} \varphi_{LJ}^\ell dz^{L}\otimes dz^{J}.
$$
Then we obtain
\begin{equation*}
\begin{aligned}
\bar \partial_M \varphi_{m_0+m-1}^\ell &= \sum_{j=1}^{n} \sum_{\substack{|L|=\ell+m-1 \\ |J|=m_0-\ell}} \overline Y_{j} \varphi_{LJ}^\ell dz^L \otimes dz^J \otimes \overline{\widetilde e_j }
\end{aligned}
\end{equation*}
and by \eqref{equation}
\begin{equation*}
\begin{aligned}
\bar \partial_M \varphi_{m_0+m-1}^\ell
&= - (m_0+m-2) \sum_{\mu=1}^{n} \varphi_{m_0+m-2}^\ell (e_\mu \circ \imath) \otimes \overline{\imath^{*} e_\mu} \\
&= -(m_0+m-2) \sum_{\mu,j=1}^{n} \varphi^\ell_{m_0+m-2} (e_\mu \circ \imath) \overline{b}_{\mu j} \otimes \overline{ \widetilde e_j}.
\end{aligned}
\end{equation*}
As a result, for each $j=1,\ldots, n$,
\begin{equation}\label{d-bar_1}
\sum_{\substack{|L|=\ell+m-1 \\ |J|=m_0-\ell}} \overline{Y_j} \varphi_{LJ}^\ell dz^I\otimes dz^J  = -(m_0+m-2) \sum_{\mu=1}^{n} \varphi_{m_0+m-2}^\ell \overline{b}_{\mu j} (e_\mu \circ \imath).
\end{equation}
Since we have
\begin{equation*}
\begin{aligned}
& \bar \partial_M \mathcal R_{G} (\varphi_{m_0+m-1}^\ell) = \bar \partial_{M} \left(\sum_{\mu=1}^{n} \varphi_{m_0+m-1}^\ell (e_\mu \circ \imath) \otimes \overline{\imath^{*} e_\mu } \right) \\
&= \sum_{\mu,j,L,J}  \left( \overline {Y_j} \varphi_{LJ}^\ell dz^{L} (e_\mu \circ \imath) \otimes dz^J  \otimes \overline{\tilde e_j}  \wedge  \overline{ \imath^* e_\mu}  \right)+ \sum_{\tau, j, \mu, s} \varphi_{m_0+m-1}^\ell dz_{\tau} \otimes \bar \partial_{M} \left( (A_{\mu\tau} \circ \imath) \overline{(A_{\mu j} \circ \imath)} \overline{\frac{\partial {\imath}_{j}}{\partial \zeta_s}} \right) \wedge d \bar \zeta_s,
\end{aligned}
\end{equation*}
\begin{equation*}
\begin{aligned}
&\sum_{\mu,j} \sum_{\substack{|L|=\ell+m-1 \\ |J|=m_0-\ell}} \overline {Y_j} \varphi_{LJ}^\ell dz^{L} (e_\mu \circ \imath) \otimes dz^{J}  \otimes \overline{\tilde e_j} \wedge  \overline{ \imath^* e_\mu}  \\
&= -(m_0+m-2) \sum_{\mu,j,{\eta}} \bar b_{{\eta} j} \varphi_{m_0+m-2}^\ell  (e_\mu \circ \imath) (e_{{\eta}} \circ \imath) \otimes \overline{\tilde e_j} \wedge  \overline{ \imath^* e_\mu} \\
&= -(m_0+m-2) \sum_{\mu,{\eta}} \varphi_{m_0+m-2}^\ell (e_\mu \circ \imath) (e_{{\eta}} \circ \imath) \overline{\imath^{*} e_{{\eta}}} \wedge \overline{\imath^* e_\mu}
=0
\end{aligned}
\end{equation*}
by \eqref{d-bar_1}, and
\begin{equation}\label{equation2}
\begin{aligned}
&\sum_{j,\mu,s} \bar \partial_{M} \left( (A_{\mu\tau} \circ \imath) \overline{(A_{\mu j} \circ \imath)} \overline{\frac{\partial {\imath}_{j}}{\partial \zeta_s}} \right) \wedge d \bar \zeta_s \\
&= \sum_{j,\mu,s} \bar \partial_{M} \left( (A_{\mu \tau} \circ \imath) \overline{(A_{\mu j} \circ \imath)} \right)  \wedge \overline{\frac{\partial \imath _j } {\partial \zeta_s }} d \bar \zeta_s + \sum_{j,\mu,{\eta},s} (A_{\mu \tau} \circ \imath) \overline{(A_{\mu j} \circ \imath)} \overline{\frac{\partial^2 {\imath}_j}{\partial \zeta_{{\eta}} \partial \zeta_s}} d \bar \zeta_{{\eta}} \wedge d \bar \zeta_s \\
&= \sum_{j,\mu,{\eta},\sigma,s} \left( \frac{\partial A_{\mu \tau}}{\partial \bar z_{{\eta}}} \overline{(A_{\mu j} \circ \imath)} + (A_{\mu \tau} \circ \imath) \frac{\partial \bar A_{\mu j}}{ \partial \bar z_{{\eta}}} \right) \overline{\frac{\partial \imath_{{\eta}}}{\partial \zeta_{\sigma}}} d \bar \zeta_{\sigma} \wedge \overline{\frac{\partial \imath_j}{\partial \zeta_s}} d \bar \zeta_s \\
&= \sum_{j,\mu,{\eta}} \left( \frac{\partial A_{\mu \tau}}{\partial \bar z_{{\eta}}} \overline{(A_{\mu j} \circ \imath)} + (A_{\mu \tau} \circ \imath) {\frac{\partial \overline{A}_{\mu j}}{ \partial \overline z_{{\eta}}}} \right) \overline{d \imath_{{\eta}} } \wedge \overline {d \imath_{j} }=0,
\end{aligned}
\end{equation}
one has $\bar \partial_M \mathcal{R}_{G} (\varphi^\ell_{m_0+m-1} )=0$.
Here the last equality in \eqref{equation2} holds by the same argument given in the proof of Lemma 4.12 in \cite{Lee_Seo}.

Now we claim that $\varphi_{m_0+k}^{\ell}, ~ k \leq m-1$ is an eigenfunction of $\Box_{m_0+{k},M}^{0,\ell+{k}}$. 
Denote $E_{m_0,{k}}^{\ell}$ be its eigenvalue. 
Since $\varphi_{m_0}^{\ell} \in H^0(M, {S^\ell T^*_{M}\otimes S^{m_0-\ell}N^*})$, 
one has $E_{m_0,0}^\ell=0$. 
Assume that $\varphi_{m_0+k}^\ell$ is an eigenvector of $\Box_{m_0+k, M}^{0,\ell+{k}}$ for some $k\geq 0$. By \eqref{RG2} and self-adjointness of ${\Box_{m_0+k+1,M}^{1,\ell+{k+1}}}$, we know 
\begin{equation}\label{perp}
\mathcal{R}_{G}(\varphi_{m_0+k}^\ell) \perp \ker \Box_{m_0+k+1,M}^{1,\ell+{k+1}}.
\end{equation}
Moreover, by Corollary~\ref{eigenvalue} and \eqref{perp}, we obtain
\begin{equation}\nonumber
\mathcal{R}_{G} (\varphi_{m_0+k}^{\ell}) 
= \Box^{1,\ell+{k+1}}_{m_0+k+1,M} G^1 \mathcal{R}_{G} (\varphi_{m_0+k}^{\ell})
= G^1  \mathcal{R}_{G}\big( (E_{m_0,k}^{\ell} + (\ell+k) + (m_0+k) ) \varphi_{m_0+k}^{\ell} \big)
\end{equation}
and by properties of the Green operator $G^1$, it follows that
$$
\Box^{0,\ell+{k+1}}_{m_0+k+1,M}(\varphi_{m_0+k+1}^{\ell})= (E_{m_0,k}^{\ell} + (\ell+k) + (m_0+k)) \varphi_{m_0+k+1}^{\ell}.
$$
The eigenvalue of $\varphi^\ell_{m_0+k+1}$ for $\Box^{0,\ell+{k+1}}_{m_0+k+1,M}$ is 
$$
E_{m_0,k+1}^{\ell}=E_{m_0,k}^{\ell}+(\ell+k)+(m_0+k).
$$
Hence
\begin{equation}\label{eigenvalue2}
\begin{aligned}
E_{m_0,k}^{\ell}&=(\ell+(\ell+1)+\cdots+(\ell+k-1))+(m_0+(m_0+1)+ \cdots+ (m_0+k-1)) \\
&=\frac{k(2m_0+k-1)}{2} + \frac{k(2\ell+k-1)}{2}.
\end{aligned}
\end{equation}
Now we will show that \eqref{equation} has a solution when $k=m$. By the Hodge decomposition, the solvability of \eqref{equation} follows by \eqref{perp}. 
By \eqref{RG1}, \eqref{equation}, and Corollary~\ref{eigenvalue}, we have 
\begin{equation*}
\begin{aligned}
\| \varphi_{m_0+m}^\ell \|^2 &= (m_0+m-1)^2 \langle \langle \bar \partial^* G^1 \mathcal{R}_{G} \varphi_{m_0 +m-1}^\ell, \bar \partial^* G^1 \mathcal{R}_{G}  \varphi_{m_0+m-1}^\ell \rangle \rangle \\
&= \frac{(m_0+m-1)^2}{E_{m_0,m-1}^{\ell} + (\ell+m-1)+(m_0+m-1)} \| \mathcal{R}_{G} \varphi_{m_0 +m-1}^{\ell} \|^2 \\
&=\frac{(\ell+m-1)+n}{(m_0+m-1)+1}\frac{(m_0+m-1)^2}{E_{m_0,m}^{\ell}}  \| \mathcal{R}_{G} \varphi^{\ell}_{m_0+m-1} \|^2.
\end{aligned}
\end{equation*}
Therefore, 
\begin{equation*}
\begin{aligned}
\|\varphi_{m_0+m}^\ell \|^2 &= \left( \prod_{j=1}^{m} \left( \frac{(\ell+j-1)+n}{(m_0+j-1)+1} \right)   \frac{2(m_0+m-j)^2}{(m-j+1)((2m_0+m-j)+(2\ell+m-j))} \right) \| \varphi_{m_0}^\ell \|^2 \\
&= \left(  \prod_{j=1}^m \bigg( \frac{\ell+j+(n-1) }{m_0+j} \bigg) \right) \frac{ (m_0+{\ell}-1)! \{(m_0+m-1)!\}^2 }{\{(m_0-1)!\}^2 m! (m_0+\ell+m-1)!} \| \varphi_{m_0}^\ell \|^2. 
\end{aligned}
\end{equation*}
\end{proof}

For nonnegative integer $k$ define 
$\varphi_k \in C^\infty(M, \imath^*(S^k T_{\Sigma}^*) ) \cong \bigoplus_{\mu=0}^k
C^\infty(M, {S^\mu T^*_{M} \otimes S^{k-\mu}N^*})$  and $\varphi\in \bigoplus_{k=0}^\infty C^\infty(M, \imath^*(S^kT_\Sigma^*))$ by
\begin{equation}\label{associated}
{\varphi_k := 0 \quad \text{ for } k < m_0,} \quad 
\varphi_k := \sum_{\ell=0}^{m_0} \varphi^{\ell}_{m_0+(k-m_0)}\quad
\text{ for } k \geq {m_0},\quad
\text{ and } \quad
\varphi := \sum_k\varphi_k
\end{equation}
where  $\varphi^\ell_{m_0+(k-m_0)} \in  {C^\infty \big(M, {S^{\ell+(k-m_0)} T_{M}^* \otimes S^{m_0-\ell}N^{*}} \big)}$. 
Using the frame $e=(e_1,\ldots, e_N)$ given in \eqref{e}, {we write} 
\begin{equation}\label{associated2}
\varphi_k(\zeta) {=} \sum_{|I| = k} f_{I} (\imath (\zeta))  e^I|_{\imath (\zeta)}
\end{equation}
for $\zeta\in M$.

\begin{lemma}\label{norm relation}
In the above setting, the following identity holds:
\begin{equation*}
\| \varphi_k \|^2 = \sum_{\ell=0}^{m_0} \| \varphi_{m_0+(k-m_0)}^{\ell} \|^2 \quad {\text{ for } k \geq m_0}.
\end{equation*}
\end{lemma}
\begin{proof}
Let $\{U_{\alpha} \}$ be a finite open cover of $\imath(M)$ in $\Sigma$ satisfying that $T_{\Sigma}^*$ on $U_{\alpha}$ has a local orthonormal frame $\{ e_1, \cdots, e_N \}$ such that $g=\sum_{\mu=1}^{N} e_\mu \otimes \bar e_\mu$ and {$h= \sum_{\mu=1}^{n}\imath^* e_\mu \otimes \imath^*\bar e_\mu$.} Write $L=(i_1,\cdots,i_n)$, $J=(i_{n+1}, \cdots, i_N)$ and $e^L=e_1^{i_1}\cdots e_n^{i_n}$, $e^J=e_{n+1}^{i_{n+1}} \cdots e_N^{i_N}$ accordingly. Then $\{ e^L \otimes e^J\}_{\{|L|=\ell+(k-m_0), |J|=m_0-\ell \}}$ becomes a local orthonormal frame of $S^{\ell+(k-m_0)} T_{M}^* \otimes S^{m_0-\ell} N^*$ on $\imath^{-1}(U_{\alpha})$, and locally
\begin{equation*}
\varphi_{k}^{\ell}(\zeta) = \sum_{\substack{|{L}|=\ell+(k-m_0) \\ |J|=m_0-\ell}}  f_{LJ} (\imath(\zeta)) e^L \big|_{\imath(\zeta)} \otimes e^J \big|_{\imath(\zeta)}
\end{equation*}
by \eqref{associated}, where $f_{LJ}$ is a smooth function on $U_{\alpha}$. Since $\imath$ is an embedding, $\{ \imath^{-1} (U_{\alpha}) \}$ becomes a finite open cover of $M$. Let $\{ \chi_{\alpha} \}$ be a partition of unity subordinate to $\{ \imath^{-1} (U_{\alpha}) \}$. 
Then
\begin{equation*}
\begin{aligned}
\| \varphi_k \|^2 = \int_M \langle \varphi_k ,\varphi_k \rangle dV_M &=\sum_{\alpha} \sum_{\ell=0}^{m_0} \sum_{\substack{|L|=\ell+(k-m_0) \\ |J|=m_0-\ell}}\int_{\imath^{-1} (U_{\alpha})}  \chi_{\alpha}  |(f_{LJ} \circ \imath )(\zeta) |^2 \langle e^L , e^L \rangle \langle e^J, e^J \rangle dV_M \\
&=\sum_{\alpha} \sum_{\ell=0}^{m_0}   \sum_{\substack{|L|=\ell+(k-m_0) \\ |J|=m_0-\ell}}  \frac{L! J!}{k!}  \int_{\imath^{-1} (U_{\alpha})} \chi_{\alpha}  |(f_{LJ} \circ \imath) (\zeta) |^2 dV_M.
\end{aligned}
\end{equation*}
Since 
$$
\|\varphi_k^\ell \|^2 = \sum_{\alpha} \sum_{\substack{|L|=\ell+(k-m_0) \\ |J|=m_0-\ell}} \frac{L! J!}{k!} \int_{\imath^{-1} (U_{\alpha})} \chi_{\alpha}  |(f_{LJ} \circ \imath)(\zeta)|^2 dV_M,
$$
the proof is completed.
\end{proof}
{By using \eqref{associated2}, we define a formal sum $f$ on {$\widetilde M\times \mathbb B^N$ } by}

\begin{equation}\label{f}
f(\zeta,w) :=  \sum_{|I|= 0}^\infty  f_{I} (\imath(\zeta)) (T_{\imath(\zeta)} w)^I.
\end{equation}

In view of Lemma~\ref{invariant} below, we may consider $f$ as a function on $\Omega := M\times_\rho \mathbb B^N$.

\begin{lemma}\label{invariant}
$f(\zeta,w)$ is $\Gamma$-invariant, i.e. $f(\gamma \zeta, \rho(\gamma)w) = f(\zeta,w)$ for all $\gamma \in \Gamma$, $\zeta \in {\widetilde M}$ and $w\in \mathbb B^N$.
\end{lemma}
\begin{proof}
Fix $\gamma \in \Gamma$.
There exists a unitary matrix $U_{\zeta}$ depending only on $\zeta$ satisfying \begin{equation}\label{transformation rule}
T_{\rho(\gamma)(\imath(\zeta))} \rho(\gamma) w = U_{\zeta} T_{\imath (\zeta)}w.
\end{equation}
Since $\varphi\in \bigoplus_{k=0}^\infty H^0 \big(M, \imath^*(S^{k}T^*_\Sigma ) \big) \cong \bigoplus_{k=0}^{\infty} H^0 \big( \imath(M), S^{{k}} T_{\Sigma}^* |_{{\imath (M)}} \big)$ and $\imath (M) \cong \imath ({\widetilde M}) / \rho(\Gamma)$,
we have $\rho(\gamma)^* \varphi_k = \varphi_k$. Note that
\begin{equation}\nonumber
\rho(\gamma)^* e_j = \sum_k  A_{jk}\circ\rho(\gamma) d(\rho \circ \gamma)_k
= \sum_{k,m,l} A_{jk}\circ\rho(\gamma) \frac{\partial (\rho \circ \gamma)_k }{\partial {\zeta}_l} A^{lm} e_m
\end{equation}
where $(A^{lm})$ denotes the inverse matrix of $A$, i.e. for {$e=(e_1,\ldots, e_N)$}
 \begin{equation*}
{\rho(\gamma)^* e\,|_{\imath(\zeta)}= \left(A\circ\rho(\gamma ) d\rho(\gamma) A^{-1}\right)^* e\,|_{\imath(\zeta)} = U_{\zeta}^* e |_{\imath(\zeta)}.}
\end{equation*}
This implies
\begin{equation}\label{invariance}
\begin{aligned}
 \sum_{|I| = 0}^\infty f_{I}(\imath(\zeta)) e^I
= \sum_{|I| = 0}^\infty f_{I}( \rho(\gamma) \imath(\zeta)) {\rho(\gamma)^*(e^I)}
= \sum_{|I| = 0}^\infty f_{I}( \rho(\gamma) \imath(\zeta)) {U_{\zeta}^*(e^I)}
\end{aligned}
\end{equation}
where $\rho(\gamma)^*$ and $U_\zeta^*$ are understood as the pull-back of symmetric differential forms. Hence by \eqref{transformation rule} and \eqref{invariance} we have
\begin{equation}\nonumber
\begin{aligned}
f(\gamma \zeta, \rho(\gamma) w)
&=  \sum_{|I| = 0}^\infty f_{I}(\imath(\gamma \zeta)) (T_{\imath(\gamma \zeta)} \rho(\gamma) w )^I
= \sum_{|I| = 0}^\infty f_{I}(\rho(\gamma)\imath( \zeta)) (T_{\rho(\gamma)\imath(\zeta)} \rho(\gamma) w )^I\\
&=  \sum_{|I| = 0}^\infty f_{I}(\rho(\gamma)\imath(\zeta)) (U_{\zeta} T_{\imath(\zeta)}w)^I
=\sum_{|I| = 0}^\infty f_{I}(\imath(\zeta)) ( T_{\imath(\zeta)}w)^I
= f(\zeta,w).
\end{aligned}
\end{equation}
\end{proof}

\subsection{$L^2$ convergence of formal series}

Let $\Omega:= M \times_{\rho} \mathbb{B}^N$ and $K\colon \mathbb{B}^N \times \overline{\mathbb{B}^N} \rightarrow \mathbb{C}$ be the (normalized) Bergman kernel on $\mathbb{B}^N$, i.e.
\begin{equation*}
K(z,w) = \frac{1}{(1- z \cdot \overline w )^{N+1}}
\end{equation*}
where $z \cdot \overline w = \sum_{i=1}^{N} z_i \overline{w_i} $.
We define a K\"ahler form $\omega$ on $\Omega$ by
\begin{equation}\nonumber
\omega|_{[\zeta, w]} = {\widetilde H}  +  \frac{\sqrt{-1}}{N+1} \partial \bar \partial \log K(w,w)
\end{equation}
with the K\"ahler form {$\widetilde H$ for $(\widetilde M, \imath^* g_{\mathbb{B}^N})$}, {where $\imath^* g_{\mathbb{B}^N}$ is the pull-back metric on $\widetilde M$ of the normalized Bergman metric $g_{\mathbb{B}^N}$ of $\mathbb{B}^N$.}
One can check that $\omega$ is an $(1,1)$ form on $M \times_{\rho} \mathbb B^N$. We define the volume form on $\Omega$ by {$dV_{\omega} = \frac{1 }{(N+n)!}\omega^{N+n} $.}
Then
\begin{equation}\label{volume}
dV_\omega = {  \sqrt{-1}^N} K(w,w) {\widetilde{H}^n} \wedge dw \wedge d\overline w,
\end{equation}
where $dw:= dw_1 \wedge \cdots \wedge dw_{{N}}$.
Now for measurable sections $f_1$, $f_2$ on $\Lambda^{p,q} T_{\Omega}^{*}$ and $\alpha > -1$,
we set
\begin{equation}\nonumber
\langle \langle f_1,f_2 \rangle \rangle_{\alpha} : = c_{\alpha} \int_{\Omega} \langle f_1, f_2 \rangle_{\omega} \delta^{\alpha+N+1} \,dV_\omega
\end{equation}
where $c_{\alpha} = \frac{\Gamma(N+\alpha+1)}{ \Gamma(\alpha+1) N! }$ and $\delta = 1-|T_{\imath (\zeta)} w|^2$.
\begin{lemma}\label{inner product}
If $f_1, f_2$ are measurable sections on $\Lambda^{p,q} T_{\Omega}^{*}$, then
\begin{equation*}
\langle \langle f_1, f_2 \rangle \rangle_{\alpha} 
= {
c_\alpha {\sqrt{-1}^N}
}\int_{\Omega} \langle f_1, f_2 \rangle_{\omega} \delta^{\alpha} \frac{|K(w, \imath (\zeta))|^2}{K(\imath (\zeta), \imath (\zeta))} {{\widetilde{H}}}^n \wedge dw \wedge d \bar w.
\end{equation*}
\end{lemma}
\begin{proof}
By \eqref{volume} and
$$
1-|T_z w|^2 = \frac{(1-|z|^2)(1-|w|^2)}{|1-z \cdot \bar w|^2} = \left( \frac{K(z,z)K(w,w)}{|K(z,w)|^2} \right)^{-\frac{1}{N+1}},
$$
the lemma follows.
\end{proof}
For $\alpha > -1$, we define a weighted $L^2$-space by setting
\begin{equation}\nonumber
L^2_{(p,q), \alpha}(\Omega) := \{ f : f \text{ is a measurable section on $\Lambda^{p,q} T_{\Omega}^*$}, ~\|f \|^2_{\alpha} := \left< f,f\right>_{\alpha} < \infty \}
\end{equation}
and a weighted Bergman space by
$A^2_{\alpha}(\Omega) := L^2_{(0,0), \alpha}(\Omega) \cap \mathcal O(\Omega).$ In this setting, we extend $\bar \partial$-operator on $\Omega$ as the maximal extension of $\bar \partial$ on $\Omega$ which acts on smooth $(p,q)$ forms on $\Omega$.

\begin{lemma}\label{norm of f_l}
For any partial sum 
$$
F_{m_0+m} (\zeta, w) = \sum_{|I|=0}^{m_0+m} f_I (\imath (\zeta)) (T_{\imath (\zeta)} w)^I
$$
of $f$ in \eqref{f}, the following identity holds:

\begin{equation}\label{L2 norm}
\|F_{m_0+m} \|^2_\alpha = \frac{2^N \pi^{N}}{N!} \sum_{|I|=0}^{m_0+m} \|\varphi_{|I|}\|^2
\frac{ |I|! \Gamma(N+\alpha+1) }{\Gamma(N+|I|+\alpha+1)}.
\end{equation}
\end{lemma}

\begin{proof}
Let $\widehat M$ denote the fundamental domain of $M $ in $\widetilde M$
and $\widetilde \Omega$ denote the corresponding domain of $\Omega$. Note that $\widetilde \Omega = \widehat M \times \mathbb{B}^N \subset \Omega =  M \times_{\rho} \mathbb{B}^N$.
By Lemma~\ref{inner product}, $\|F_{m_0+m} \|^2_{\alpha}$ is equal to
\begin{equation}\label{norm of f_1}
\begin{aligned}
{c_{\alpha} 2^N}  \int_{{\widetilde \Omega}}  \bigg|\sum_{|I|= 0}^{m_0+m}  f_{I} (\imath (\zeta) ) (T_{\imath(\zeta)} w)^I\bigg|^2 ( 1-|T_{\imath(\zeta)} w |^2)^{\alpha}  {{\widetilde{H}}}^n   \frac{|K(w, \imath (\zeta) )|^2}{K(\imath(\zeta), \imath (\zeta) )} d\lambda_{w}. \end{aligned}
\end{equation}
where $\lambda_w = (\frac{\sqrt{-1}}{2})^N dw_1 \wedge d \overline w_1 \wedge \cdots \wedge dw_N \wedge d \overline w_N$ denotes the Lebesgue measure of $\mathbb{B}^N$.

Since $t=T_{\imath (\zeta) } w$,
$ J_{\mathbb R} T_{\imath (\zeta)} (0)
 = (1-|\imath (\zeta) |^2)^{N+1}$, $d\lambda_w
= | J_{\mathbb C}T_{\imath (\zeta)} t|^2 d\lambda_t$,

\begin{equation}\nonumber
K(\imath (\zeta) ,w) = K(T_{\imath (\zeta)} 0, T_{\imath (\zeta)} t) = \frac{ K(0,t) }{J_{\mathbb C}T_{\imath(\zeta)} (0) \overline{ J_{\mathbb C} T_{\imath (\zeta) }(t)}}
= \frac{ 1}{J_{\mathbb C}T_{\imath (\zeta)} (0) \overline{ J_{\mathbb C} T_{\imath (\zeta)} (t)}},
\end{equation}
and
\begin{equation}\nonumber
K(\imath (\zeta), \imath (\zeta) ) = K(T_{\imath (\zeta)} 0, T_{\imath (\zeta)} 0)  = \frac{K(0,0)}{| J_{\mathbb{C}} T_{\imath (\zeta)} (0) |^2} = \frac{1}{|J_{\mathbb{C}} T_{\imath (\zeta) }  (0) |^2},
\end{equation}
by \eqref{norm of f_1} we obtain
\begin{equation}\label{norm of f_2}
\begin{aligned}
\|F_{m_0+m}\|^2_{\alpha}
& = 2^N c_\alpha \int_{{\widehat M}}  {{\widetilde{H}}}^n \int_{\mathbb{B}^N}
\bigg|\sum_{|I|= 0}^{m_0+m}    f_{I} (\imath (\zeta) ) t^I\bigg|^2 ( 1-|t|^2 )^{\alpha}   d\lambda_{t} \\
&= 2^N c_\alpha      \int_{\widehat M} {{\widetilde{H}}}^n      \int_{\mathbb{B}^N} \sum_{|I|= 0}^{m_0+m}
\left| f_I (\imath (\zeta)) t^I \right|^2 \big( 1-|t|^2  \big)^{\alpha} d \lambda_{t}  \\
\end{aligned}
\end{equation}
The second equality in \eqref{norm of f_2} can be induced by the orthogonality of polynomials
with respect to the inner product $\int_{\mathbb{B}^n} f \bar g (1-|t|^2)^{\alpha} d \lambda_{t}$ (\cite{Zhu}).

Since
\begin{equation}\nonumber
\begin{aligned}
\| \varphi_{\ell} \|^2
= \sum_{|I|=\ell} \frac{I!}{\ell!}   \int_{{\widehat M}} {\left|f_{I} \circ \imath \right|^2} {{\widetilde{H}}}^n,
\end{aligned}
\end{equation}
by \eqref{norm of f_2} one has
\begin{equation}\nonumber
\begin{aligned}
\|F_{m_0+m} \|^2_{\alpha} &= \frac{2^N \pi^{N}}{N!} \sum_{|I|=0}^{m_0+m} \|\varphi_{|I|}\|^2 \frac{|I|!
\Gamma(N+\alpha+1) }{\Gamma(N+|I|+\alpha+1)}.
\end{aligned}
\end{equation}
\end{proof}

\begin{lemma}\label{convergence}
For any $\alpha >-1$, the formal sum $f$  converges in $L^2_{\alpha} (\Omega)$. Moreover, if $n \not = N$, then $f$ converges in $L^2_{-1}(\Omega)$.
\end{lemma}
\begin{proof} 
By {\eqref{L2 norm}} and Lemma~\ref{norm relation}, the partial sum
$$
F_{m_0+m} (\zeta, w) = \sum_{|I|=0}^{m_0+m} f_{I} (\imath(\zeta)) (T_{\imath(\zeta)} w)^I
$$
satisfies
\begin{equation*}
\begin{aligned}
\| F_{m_0 +m} \|^2_{\alpha} &= \frac{{2^N}\pi^N}{N!} \sum_{k=0}^{m} \| \varphi_{m_0+k} \|^2 \frac{(m_0+k)! \Gamma(N+\alpha+1)}{\Gamma(N+ m_0+k + \alpha +1 )} \\
&= \frac{{2^N}\pi^N}{N!} \sum_{k=0}^{m} \left( \sum_{\ell=0}^{m_0} \| \varphi_{m_0+k}^{\ell} \|^2 \right) \frac{(m_0+k)!\Gamma(N+\alpha+1)}{\Gamma(N+m_0+k + \alpha +1 )}  \\
&= \frac{{2^N}\pi^N}{N!} \frac{\Gamma(N+\alpha+1) \Gamma(m_0+1)}{\Gamma(m_0+N+\alpha+1)} {\sum_{\ell=0}^{m_0} \sum_{k=0}^{m}  a_k^\ell},
\end{aligned}
\end{equation*}
where
\begin{equation}\label{term of series}
{a_k^\ell} := \frac{(m_0+1)_k}{(N+m_0+\alpha+1)_k} \frac{ (m_0)_k (m_0)_k}{(m_0+\ell)_k} \frac{1}{k!} \bigg( \prod_{j=1}^{k} \left(\frac{(\ell+j)+(n-1)}{m_0+j} \right) \bigg)\| \varphi_{m_0}^\ell \|^2
\end{equation}
and $(m_0)_{k} := m_0(m_0+1) \cdots (m_0+k-1)$.
Note that for each fixed $\ell=0,\cdots,m$,
\begin{equation*}
\begin{aligned}
k \left( \frac{a_k^\ell}{a_{k+1}^\ell} -1 \right) &= k \left( \frac{(k+1)(m_0+\ell+1)(k+N+m_0+\alpha+1)   }{ (k+m_0)^2(k+n+\ell)   } -1 \right) \\
&=k\left(\frac{(\ell-m_0+1)k-{m_0}^2}{(k+m_0)^2} \right) + (N+m_0+\alpha+1-(n+\ell)) \frac{k(k+1)(m_0+\ell+1)}{(k+m_0)^2 (k+n+\ell)} \\
& \rightarrow (\ell-m_0+1)+(N+m_0+\alpha+1-n-\ell) = 1+(N-n+\alpha+1)
\end{aligned}
\end{equation*}
as $k \rightarrow \infty$. Hence the series $\sum_{k=0}^{\infty}  a_k^\ell$ converges when $\alpha>n-(N+1)$ by the Raabe's test. Since $F_{m_0+m}$ is the partial sum of $f$, the lemma is now proved.
\end{proof}

\begin{remark}
Set $\alpha<-1$. For any formal sum $f$ given by  \eqref{f} define {$L^2$-norm for any partial sum of $f$} by \eqref{L2 norm}. Then Lemma~\ref{convergence} tells us that $f$ converges in $L^2_\alpha(\Omega)$ if $\alpha>n-(N+1)$ (cf. \cite[Chapter 12]{ZZ08}).
\end{remark}

\subsection{Holomorphicity of formal series}

\begin{lemma}\label{holo}
The formal sum $f$ given by \eqref{f}
is holomorphic.
\end{lemma}

\begin{proof}
Note that since $f$ is holomorphic in $w$, we only need to show that $f$ is holomorphic in $\zeta$.
Let
$$
F_m(\zeta,w) := \sum_{|I|= 0}^m   f_{I}(\imath (\zeta) ) (T_{\imath (\zeta)} w)^I
$$
be the finite sum of $f$ and let $\Gamma_l^{j\mu} := \sum_{k,s}\overline A^{kj} \frac{\partial A_{ls}}{\partial \overline z_k} A^{s\mu}$.
Since
$$
\frac{\partial F_{m}}{\partial \bar \zeta_j } (\zeta,w)
= \frac{\partial \widetilde F_m}{\partial \bar \zeta_j} (\zeta,T_{\imath (\zeta)} w)
+ \sum_{k,\nu} \frac{\partial \widetilde F_m} {\partial t_k} (\zeta,T_{\imath (\zeta)} w) \frac{ \partial (T_{z} w)_{k} } {\partial \bar z _\nu} \bigg|_{z=\imath(\zeta) 
}\frac{\partial \bar \imath_{\nu} }{ \partial \bar \zeta_j }
$$
with $\widetilde F_m(\zeta ,t) := \sum_{|I|= 0}^m   f_I (\imath (\zeta)) t^I$, we obtain
\begin{equation}\nonumber
\begin{aligned}
 \overline Y_{\mu} F_m &= \sum_{j} \bar y_{j \mu} \frac{\partial \widetilde F_m}{\partial \bar \zeta_j} (\zeta,T_{\imath(\zeta)}w) + \sum_{j} \bar y_{j \mu} \left(  \sum_{k,\nu} \frac{\partial \widetilde F_m} {\partial t_k} (\zeta,T_{\imath (\zeta)} w) \frac{ \partial (T_{z} w)_{k} } {\partial \bar z_\nu} \bigg|_{ z=\imath(\zeta) 
} \frac{\partial \bar \imath_{\nu} }{ \partial \bar \zeta_j } \right) \\
&=\sum_{j} \bar y_{j \mu} \frac{\partial \widetilde F_m}{\partial \bar \zeta_j} (\zeta,T_{\imath(\zeta)}w) + \sum_{\tau=1}^{N} \overline {(a_{\tau \mu} \circ \imath)} \sum_{k, \nu} \bar{A}^{\nu \tau} \left( \frac{\partial \widetilde F_m}{\partial t_k} (\zeta,T_{\imath (\zeta)} w) \frac{\partial {(T_z w)}_{k}} {\partial \bar z_{\nu}} \bigg|_{ z=\imath(\zeta)
} \right)  \\
&= \sum_{|I|= 0}^m \bigg( \overline{Y}_{\mu} (f_I \circ \imath)  + \sum_{\tau=1}^{N} \overline{(a_{\tau \mu} \circ \imath)} \bigg( \sum_{k=1}^{N} i_k (f_{I} \circ \imath) \Gamma_{k}^{\tau k} \\
&\quad\quad\quad
+ \sum_{k=1}^{N} (i_k +1) \sum_{q \not =k} (f_{i_1 \cdots i_{k+1} \cdots i_{q-1} \cdots i_N} \circ \imath ) \Gamma_{k}^{\tau q}
+ |I| (f_I \circ \imath) (T_{\imath (\zeta)} w)_{\tau} \bigg)  \bigg)  (T_{\imath(\zeta)}w)^I.
\end{aligned}
\end{equation}
Here, the second equality holds by Lemma~\ref{useful formula} and the third equality holds by the equation (4.8) and Lemma 4.8 in \cite{Lee_Seo}. If we express $\varphi_{s} = \sum_{|I|=s} f( \imath (\zeta))  e^I |_{\imath (\zeta)}$, then we have
\begin{equation}\label{dbar1}
\begin{aligned}
\bar \partial \varphi_{s}
&=  \sum_{|I| =s} \sum_{\mu=1}^{n} \bigg( \overline Y_{\mu} (f_{I} \circ \imath) + \sum_{\tau=1}^{N} \overline{(a_{\tau \mu} \circ \imath)} \bigg( \sum_{k} i_k (f_{I} \circ \imath) \Gamma_{k}^{\tau k}\\
& \quad\quad\quad\quad\quad\quad\quad
 + \sum_{k} (i_k+1) \sum_{q \not = k} (f_{i_1 \cdots i_k+1 \cdots i_q-1 \cdots i_N} \circ \imath) \Gamma_{k}^{\tau q} \bigg)   \bigg) e^I |_{\imath (\zeta)}  \otimes \overline{{ \widetilde e }_{\mu}}.
\end{aligned}
\end{equation}
On the other hand, one has
\begin{equation}\label{dbar2}
\begin{aligned}
\bar \partial \varphi_{s}
&= -(s-1) \mathcal R_G( \varphi_{s-1} ) \\
&= -(s-1) \sum_{\mu=1}^n \sum_{\tau=1}^{N} \overline{( a_{\tau \mu} \circ \imath)}  \sum_{|J| = s-1}  (f_{J} \circ \imath) {(e^J e_\tau )|_{\imath (\zeta)}} \otimes \overline{ \widetilde{ e_{\mu} } }
\end{aligned}
\end{equation}
{by \eqref{cases}, \eqref{associated}, and the definition of $\mathcal{R}_{G}$}  with Lemma~\ref{ab}.
Hence by comparing \eqref{dbar1} and \eqref{dbar2} one obtains
\begin{equation}\nonumber
\begin{aligned}
-(s-1) & \sum_{|J| = s-1} \sum_{\tau=1}^{{ N}} \overline{( a_{\tau \mu} \circ \imath)} (f_{J} \circ \imath) t^J t_{\tau} = \sum_{|I| =s} \bigg( \overline Y_{\mu} (f_{I} \circ \imath)\\
&+ \sum_{\tau=1}^{N} \overline{(a_{\tau\mu} \circ \imath)} \bigg( \sum_{k} i_k (f_{I} \circ \imath) \Gamma_{k}^{\tau k}  + \sum_{k} (i_k+1) \sum_{q \not = k} (f_{i_1 \cdots i_k+1 \cdots i_q-1 \cdots i_N} \circ \imath) \Gamma_{k}^{\tau q} \bigg) \bigg) t^I.
\end{aligned}
\end{equation}
Therefore we obtain
$$
\overline Y_\mu F_m =  m \sum_{|I| = m} \sum_{\tau=1}^{N} \overline{ (a_{\tau \mu} \circ \imath)} (f_I \circ \imath) (T_{\imath (\zeta)} w)^I (T_{\imath(\zeta)}w)_{\tau}.
$$
If $f_1$ and $f_2$ are monomials in $t$ with $f_1 \neq cf_2$ for any $c\in \mathbb R$, we have $\int_{\mathbb{B}^n}
f_1 \bar f_2 (1-|t|^2)^{\alpha} d \lambda_{t}=0$.
Hence one obtains
\begin{equation}\nonumber
\begin{aligned}
\| \bar \partial F_{m} \|^2_{1}
& = m^2 \sum_{\tau=1}^{n} \bigg\| \sum_{|I| =m}  \overline{( a_{\tau \mu} \circ \imath )} (f_{I} \circ \imath)(\zeta) (T_{\imath(\zeta)} w)^I (T_{\imath(\zeta)}w)_{\tau} \bigg\|^2_{1}\\
& \lesssim m^2 \sum_{|I| =m} \int_{{\widehat{M}}} |(f_{I} \circ \imath) (\zeta)|^2 {{\widetilde H}}^n
\bigg( \sum_{\tau=1}^{n} \int_{\mathbb{B}^N} | t^I t_{\tau}|^2 (1-|t|^2) d\lambda_{t} \bigg) \\ 
& \lesssim m^2 \sum_{|I|=m} \| \varphi_I \|^2  \frac{m! \Gamma(N+2) \pi^N (i_1 + \cdots + i_n +n) }{N! \Gamma(N+m+3)}  \\
& \lesssim m^2 \sum_{|I|=m} \| \varphi_I \|^2  \frac{m! \Gamma(N+2) \pi^N (m +n) }{N! \Gamma(N+m+3)}\\
&\lesssim m^2  \sum_{\ell=0}^{m_0}  \| \varphi_m^\ell \|^2 \frac{m! \Gamma(N+2) \pi^N (m+n) }{N! \Gamma(N+m+3)}\\
&  \lesssim \sum_{\ell=0}^{m_0} \left( \frac{1}{\left(\frac{N}{m}+1+\frac{2}{m}\right) \left(\frac{N}{m} +1 + \frac{1}{m}\right)}  \frac{m!(m +n)}{(m+N)!} \right)  \| \varphi^\ell_{m} \|^2
\end{aligned}
\end{equation}
for $m \geq m_0$ by using Lemma~\ref{norm relation}.

Note that 
\begin{equation*}
\frac{1}{\left(\frac{N}{m}+1+\frac{2}{m}\right) \left(\frac{N}{m} +1 + \frac{1}{m}\right)} \frac{m!(m+n)}{(m+N)!}  = O\big(m^{-(N-1)} \big).
\end{equation*}
Moreover, 
\begin{equation*}
\frac{(m_0+\ell-1)! \{ (m_0+(m-{m_0})-1)!\}^2 }{\{(m_0-1)!\}^2 (m-{m_0})! (m_0+\ell+(m-{m_0})-1)!  }= O(m^{-1 + (m_0-\ell)})
\end{equation*}
and
\begin{equation}\nonumber
 \prod_{j=1}^{m-{m_0}} \left(\frac{(\ell+j)+(n-1)}{m_0+j} \right) = \frac{m_0! (\ell+(m-{m_0})+(n-1))!}{ (\ell+(n-1))! \big( m_0+(m-{m_0}) \big)! } = O(m^{(n-1)-(m_0-\ell)})
\end{equation}
by Lemma~\ref{existence} and Stirling's formula.
Hence $\| \bar \partial F_{m} \|^2_{1}= O(m^{n-N-1}) \rightarrow 0$ as $m \rightarrow \infty$. Therefore by the distribution theory, we conclude that $f$ is holomorphic.
\end{proof}

Let $f$ be a holomorphic function on $\Omega = M \times_{\rho} \mathbb{B}^N $. {Using $\imath$, we may regard $M \times_{\rho} \mathbb{B}^N$ as a quotient of $\imath (\widetilde{M}) \times \mathbb{B}^N$ under the diagonal action of $\rho(\Gamma)$ and it becomes a complex submanifold of $\mathbb{B}^N \times \mathbb{B}^N / \rho(\Gamma)$ which is a quotient of $\mathbb{B}^N \times \mathbb{B}^N$ under the same action.} So we may {identify} $f \in \mathcal{O} (\Omega)$ with $f \in \mathcal{O}( \imath ({\widetilde M}) \times \mathbb{B}^N)$ which satisfies $f(\imath (\zeta), w) = f ( (\rho \circ \gamma) (\imath (\zeta) ) , (\rho \circ \gamma) (w)  )$ for any $\gamma \in \Gamma$.

Let $(z,w) \in \imath({\widetilde M}) \times \mathbb{B}^N \subseteq \mathbb{B}^N \times \mathbb{B}^N$. Since $\tilde f (z,t) := f(z, T_z t) = f(z,w)$ is holomorphic for $t=T_z w$, we may express $ \tilde f$ by
$$
\tilde f(z,t) = \sum_{|I|=0}^{\infty} f_I (z) t^{I},\quad \text{where} \quad f_I(z) = \frac{1}{I!} \frac{\partial^{| I |} \tilde f}{\partial t^I} (z,0) \in C^{\infty} (\imath ({\widetilde M}) )
$$
Hence
$$
f(\imath (\zeta), w) = \sum_{|I|=0}^{\infty} f_I (\imath (\zeta)) (T_{\imath (\zeta)} w)^I
$$
on ${\widetilde M}\times \mathbb{B}^N$.
We associate $\sum_{|I|=0}^{\infty} f_I (\imath (\zeta)) (T_{\imath (\zeta)} w)^I$ to a set of sections $\{ \varphi_m\}$ with  ${\varphi_m} \in C^{\infty} (M, \imath^* (S^{m} T_{\Sigma}^*))$
which is defined by
$$
\varphi_m :=   \sum_{|I|=m} f_I (\imath (\zeta) ) {e^I \big|_{\imath (\zeta)}}
$$
where ${e^I}= e_1 ^{i_1} \cdots e_N^{i_N}$ and $i_1 +\cdots + i_N = m$. We call $\{ \varphi_m \}$ the {\it associated differential} of $f$ on $M$. Note that by a similar argument of Lemma~\ref{norm of f_l}, we obtain
\begin{equation}\label{norm of f}
\| f \|^2_{\alpha} = \frac{{2^N}\pi^N}{N!} \sum_{|I|=0}^{\infty} \| \varphi_{|I|} \|^2 \frac{|I|! \Gamma(N+\alpha+1)}{\Gamma(N+|I|+\alpha+1)} 
\end{equation}

The Hardy space $A^2_{-1}(\Omega)$ is defined by
$$
A^2_{-1}(\Omega) := \{  f \in \mathcal{O}(\Omega) : \|f \|^2_{-1} < \infty  \}$$
where the norm $\| f\|^2_{-1}$ is given by
\begin{equation}\label{Hardy norm}
\|f \|^2_{-1} := \frac{{2^N}\pi^N}{N!} \sum_{|I|=0}^{\infty} \| \varphi_{|I|} \|^2 \frac{|I|! \Gamma(N)}{\Gamma (N+|I|)},
\end{equation}
with the associated differential  $\{\varphi_{|I|} \}$ of $f$.

\begin{lemma}\label{Hardy}
If $N>n$, then for any $\alpha >-1$, $A^2_{-1}(\Omega) \subset A^2_{\alpha} (\Omega)$.
\end{lemma}
\begin{proof}
For any $|I| \geq 1$ the inequality
\begin{equation*}
\frac{|I|! \Gamma(N+\alpha+1)}{\Gamma(N+|I|+\alpha+1)} < \frac{|I|! \Gamma(N)}{\Gamma(N+m)}
\end{equation*}
is equivalent to
\begin{equation}\label{inequality}
\frac{(\alpha+1)_{N}}{(\alpha+1)_{N+|I|}} < \frac{(N-1)!} {(N+|I|-1)!}
\end{equation}
and \eqref{inequality} holds whenever $\alpha>-1$. Therefore, the lemma follows from  {\eqref{norm of f}}, \eqref{Hardy norm} and the comparison test.
\end{proof}

Now we define a linear map
$$
{\Phi: \bigoplus_{m=0}^{\infty}H^{0}(M, \imath^*( S^m T_{\Sigma}^* ))   \rightarrow \mathcal{O} (\Omega)},
$$
For a constant function $\psi \in H^{0}(M, \imath^* (S^0 T_{\Sigma}^*))$, identifying $S^0 T^*_\Sigma$ with the trivial line bundle $\Sigma\times\mathbb C$, we associate $\psi$ to the constant function $\Phi(\psi)$ of the same constant value.
For a non-zero $\psi \in H^0(M, \imath^*(S^mT^*_
\Sigma))\cong\bigoplus_{\ell=0}^mH^0 (M, {S^{\ell} T_{M}^* \otimes S^{m-\ell} N^*})$, {we consider  sequences $\{\varphi_k^0 \}, \cdots ,\{\varphi_k^m \}$ for $\psi$ described in  \eqref{cases} and Lemma~\ref{existence} and define $\Phi(\psi)$ by the formal sum $f$ given by \eqref{associated}, \eqref{associated2}, and 
\eqref{f}.} Then by \eqref{norm of f} and Lemma~\ref{convergence}, the image of $\Phi$ is contained in $A^2_{\alpha}(\Omega)$ for any $\alpha >-1$. If $n \not = N$, then $\Phi(\psi)$ belongs to $A^2_{-1}(\Omega)$.

\begin{lemma}\label{symm diff from holo}
Let $f$ be a holomorphic function on $\Omega$. Then the associated differential $\{\varphi_m \}$ of $f$ on $M$ satisfies
\begin{equation*}\label{recursive}
\bar \partial_{M} \varphi_{m} =-(m-1)\mathcal{R}_{G} \varphi_{m-1}.
\end{equation*}
\end{lemma}
\begin{proof}
Take a point $q \in \widetilde{M}$ and small open set $q \in U \subset \widetilde{M}$ such that $U \cong \imath (U) \subset \imath (\widetilde M)$. Consider a local coordinate system $(z_1,\cdots, z_N)$ at $p :=\imath(q) \in \imath(U)$ such that 
$$
\imath(U) = \{(z_1,\cdots, z_N): z_{n+1}=\cdots = z_N=0 \} \quad \text{near} \quad p=(0,\cdots,0).
$$
Since the holomorphicity of $f$ on $\Omega$ implies
\begin{equation}\label{equation3}
0 = {\frac{\partial}{\partial \bar \zeta_j} f( \imath (\zeta), w) } =  \sum_{\mu=1}^{n} \frac{\partial \tilde f}{\partial \bar z_{\mu}}  \frac{\partial \bar \imath_{\mu}}{\partial \bar \zeta_j} 
+ \sum_{k=1}^{N} \sum_{\mu=1}^{n} \frac{\partial \tilde f}{\partial t_k} \frac{\partial(T_{ z} w)_k}{\partial \bar z_{\mu}} \bigg |_{z=\imath(\zeta)} \frac{\partial \bar \imath_{\mu}} {\partial \bar \zeta_j} \quad \text{for} ~~ {j=1,\cdots,n},
\end{equation}
by substituting $w=T_{\imath (\zeta)} t$ into \eqref{equation3}, we have
$$
0 =  \sum_{\mu=1}^{n} \frac{\partial \tilde f}{\partial \bar z_{\mu}} \frac{\partial \bar \imath_{\mu}}{\partial \bar \zeta_j} 
+ \sum_{k=1}^{N} \sum_{\mu=1}^{n} \frac{\partial \tilde f}{\partial t_k} \frac{\partial(T_{ z} w)_k}{\partial \bar z_{\mu}} \bigg|_{\begin{subarray}{l} z=\imath(\zeta) \\ w=T_{\imath(\zeta)}t \end{subarray}
}  \frac{\partial \bar \imath_{\mu}}{\partial \bar \zeta_j} \quad \text{for} ~~ {j=1,\cdots,n}.
$$
Therefore
\begin{equation*}
\begin{aligned}
0 &= \frac{1}{I!} \frac{\partial^{|I|} } {\partial t^{I}} \bigg |_{t=0} \left(\sum_{\mu=1}^{n}  \frac{\partial \tilde f}{\partial \bar z_{\mu}} \frac{\partial \bar \imath_{\mu}}{\partial \bar \zeta_j} + \sum_{k=1}^{N} \sum_{\mu=1}^{n} \frac{\partial \tilde f}{\partial t_k}  \frac{\partial(T_{ z} w)_k}{\partial \bar z_{\mu}}\bigg|_{\begin{subarray}{l} z=\imath(\zeta) \\ w=T_{\imath(\zeta)}t \end{subarray}
} \frac{\partial \bar \imath_{\mu}}{\partial \bar \zeta_j} \right) \\
&=   \sum_{\mu=1}^{n}  \frac{\partial \bar \imath_\mu}{\partial \bar \zeta_j} \frac{\partial f_I  }{\partial \bar z_{\mu}} \bigg |_{z= \imath(\zeta) }  +  \sum_{k=1}^{N} \sum_{\mu=1}^{n} \frac{1}{I!} \frac{\partial^{|I|}}{\partial t^I} \bigg |_{t=0} \left(  \frac{\partial \tilde f}{\partial t_k}  \frac{\partial(T_{ z} w)_k}{\partial \bar z_{\mu}}\bigg|_{\begin{subarray}{l} z=\imath(\zeta) \\ w=T_{\imath(\zeta)}t \end{subarray}
} \frac{\partial \bar \imath_{\mu}}{\partial \bar \zeta_j}  \right).
\end{aligned}
\end{equation*}
By a similar computation to Proposition 4.9 in \cite{Lee_Seo} with {Lemma~\ref{useful formula}}, 
\begin{equation}\label{taylor}
\begin{aligned}
0&= \overline {Y_{q}} (f_{i_1 \cdots i_N} \circ \imath) + \sum_{l=1}^{N} \overline {( {a_{lq}} \circ \imath)} \bigg( \sum_{k=1}^{N} \bigg( i_k (f_{i_1\cdots i_N} \circ \imath) \Gamma_{k}^{lk} \\
&\quad\quad\quad\quad\quad\quad\quad + \sum_{\tau \not =k} (i_k+1) (f_{i_1 \cdots i_k +1 \cdots i_{\tau} -1 \cdots i_N} \circ \imath) \Gamma^{l \tau}_{k} \bigg)
+ (|I|-1) (f_{i_1 \cdots i_{q}-1 \cdots i_N} \circ \imath) \bigg).
\end{aligned}
\end{equation}
If we express $\varphi_{l} = \sum_{|I|=l} f_I (\imath (\zeta) ) {e^I |_{\imath (\zeta)}}$, we have
\begin{equation*}
\begin{aligned}
\bar \partial \varphi_{i_1 \cdots i_N}
&=   \sum_{q=1}^{n} \bigg( \overline {Y_{q}} (f_{i_1 \cdots i_N } \circ \imath) + \sum_{k,l=1}^{N} \overline {({a_{lq}} \circ \imath)}  i_k (f_{i_1 \cdots i_N} \circ \imath) \Gamma_{k}^{l k}  \bigg) e_{1}^{i_1} \cdots e_{N}^{i_N} \otimes \overline { {\widetilde e}_{q} }   \\
& \quad+ \sum_{q=1}^{n} \bigg( \sum_{k,l=1}^{N}  {\bar a_{lq} }  i_k \sum_{\tau \not = k} (f_{i_1 \cdots i_N} \circ \imath) \Gamma_{k}^{l \tau}    \bigg) e_1^{i_1} \cdots e_k^{i_k-1} \cdots e^{i_{\tau}+1}_{\tau} \cdots e_N^{i_N} \otimes \overline{ {\widetilde e}_{q}}
\end{aligned}
\end{equation*}
Therefore, the lemma follows by \eqref{taylor}.
\end{proof}

\begin{remark}
By Lemma~\ref{symm diff from holo}, for any $f \in \mathcal{O} (\Omega)$ which vanishes up to $m$-th order with nonvanishing $(m+1)$-th order on $D:= \{ [( \zeta, \imath(\zeta) )] \in \Omega : \zeta \in {\widetilde M} \}$, there exists a nonzero holomorphic section $\varphi_{m+1}$ of $\imath^* (S^{m+1}T_\Sigma^*)$ associated to $f$.
\end{remark}

\begin{proposition}\label{density}
The linear map
$$
\Phi: \bigoplus_{m=0}^{\infty} H^0 (M, \imath^* (S^m T_\Sigma^*)) \rightarrow \mathcal{O} (\Omega)
$$
has a dense image in $\mathcal{O}(\Omega)$ equipped with the compact open topology.
\end{proposition}

Let
$$
\Omega_{\epsilon} := \{ [(\zeta,w)] \in \Omega : |T_{\imath (\zeta)} w | < \epsilon \}
$$
with $0<\epsilon<1$. These $\Omega_\epsilon$ exhausts $\Omega$.
Define
$$
L^2(\Omega_{\epsilon}) : = \{ f : \text{$f$ is measurable function on $\Omega_{\epsilon}$ such that} ~ \| f \|^2_{0,\epsilon} := \langle \langle f, f \rangle \rangle^2_{0, \epsilon} < \infty \},
$$
where
$$
\langle \langle f , g \rangle \rangle^2_{0,\epsilon} := \int_{\Omega_{\epsilon}} f \bar g \delta^{N+1} dV_{\omega}.
$$
The Bergman space $A^2 (\Omega_{\epsilon})$ is given by $L^2 (\Omega_{\epsilon}) \cap \mathcal{O} (\Omega_{\epsilon})$.

\begin{proof} 
Since the proof of the proposition is similar to those in \cite{Adachi} and \cite{Lee_Seo}, we will only give a sketch of it. 
By the Cauchy estimate, it suffices to show that the image of $\Phi$ is dense in $A^2 (\Omega_{\epsilon})$ for any $0<\epsilon<1$. For a contradiction, suppose that there exists a non-zero holomorphic function $f \in A^2(\Omega_{\epsilon})$ which is orthogonal to the image of $\Phi$ in $A^2 (\Omega_{\epsilon})$. Then, for the associated differential $\{\varphi_k\}$ of $f$ on $M$, there exists $m_0 \in \mathbb{N}$ such that $\varphi_{k}=0$ for any $k<m_0$, but $\varphi_{m_0} \not =0$. Since $\varphi_{m_0} \not =0$, there exists an $\ell$, $0 \leq \ell \leq m_0$
such that $0 \not = \varphi_{m_0}^{\ell} \in H^0 \big(M, {S^\ell T_{M} \otimes S^{m_0-\ell} N^*} \big)$. 

Now we define orthogonal projections:
\begin{equation*}
\Pi^{i,\ell}_{m_0+m, E^{\ell}_{m_0+m}}: L^2 (M,  {S^{\ell+m} T_{M} \otimes S^{m_0-\ell} N^*} \otimes {\Lambda^{0,i} T_{M}^*}  ) \rightarrow \ker (\Box^{i,\ell}_{m_0+m,M} - E_{m_0+m}^{\ell} I)
\end{equation*}
for $i=0,1$ {where $E^\ell_{m_0,m}$ is given by \eqref{eigenvalue2}.}
Let $\{\widetilde \varphi_k^\ell \}$ be the sequence satisfying \eqref{cases} with respect to the symmetric differential $\varphi_{m_0}^\ell$.
Since $\langle \varphi_{m_0+m}, \widetilde \varphi_{m_0+m}^\ell \rangle = \langle \varphi_{m_0+m}^\ell , \widetilde \varphi_{m_0+m}^\ell \rangle$ holds for each $m$ and $\ell$, 
if we prove that the sequence $\{\widetilde \varphi^\ell_k \}$ equals to $\{ \Pi_{m_0+m, E_{m_0,m}^\ell}^{0,\ell} (\varphi^\ell_{m_0+m}) \}_{m=0}^{\infty}$, then it gives a contradiction.
For this, we will use induction.

Suppose that this claim is true for any $m \leq k-1$. Since $f-\Phi(f)$ is also holomorphic, by Lemma ~\ref{symm diff from holo} we have
$$
\bar \partial_{M} (\varphi_{m_0+k}^\ell - \widetilde \varphi^\ell_{m_0+k}) = -(m_0+k-1) \mathcal{R}_{G} (\varphi^\ell_{m_0+k-1} - \widetilde \varphi^\ell_{m_0+k-1}).
$$
If we prove 
\begin{equation}\label{final1}
\Pi^{1,\ell}_{m_0+k, E^\ell_{m_0,k}} \bar \partial (\varphi_{m_0+k}^\ell - \widetilde \varphi^\ell_{m_0+k}) = \bar \partial \big( \Pi^{0,\ell}_{m_0+k, E_{m_0,k}^\ell} ( \varphi^\ell_{m_0+k} - \widetilde \varphi^\ell_{m_0+k} ) \big)
\end{equation}
and
\begin{equation}\label{final2}
\Pi^{1,\ell}_{m_0+k, E_{m_0,k}^\ell} \mathcal R_{G} (\varphi^\ell_{m_0+k-1} - \widetilde \varphi^\ell_{m_0+k-1}) = \mathcal R_{G} (\Pi^{0,\ell}_{m_0+k-1, E^\ell_{m_0,k-1}}(\varphi^\ell_{m_0+k-1} - \widetilde \varphi^\ell_{m_0+k-1}) ),
\end{equation}
then by $\ker \bar \partial \perp \ker(\Box^{0,\ell}_{m_0+k} - E_{m_0,k}^\ell I)$, it follows that $\widetilde \varphi_{m_0+k}^\ell = \Pi_{m_0+k, E_{m_0,k}^\ell}^{0,\ell} (\varphi_{m_0+k}^\ell)$ and therefore the claim is proved. Since \eqref{final1} follows by a straightforward computation and \eqref{final2} follows by Corollary~\ref{eigenvalue} and the assumption, the proof is completed.
\end{proof}

\begin{proof}[Proof of Theorem 1.1]
To show that $\Phi$ is injective, since 
$\Phi(H^0(M, {S^\ell T^*_{M}\otimes S^{m-\ell}N^*}))$ 
are orthogonal to each other if $m$ or $\ell$ are different by Lemma~\ref{existence}, we only need to consider when $\psi_1, \psi_2$ belong to $H^0(M,\imath^* (S^m T_{\Sigma}^*))$ such that $\psi_1 \not = \psi_2$. 
However in this case $\Phi(\psi_1)$ and $\Phi(\psi_2)$ are different by the construction \eqref{f}.
By Proposition~\ref{density} and {Lemma~\ref{holo}}, the proof is completed.
\end{proof}
\bigskip

\noindent
{\bf \Large Declarations}
\bigskip

\noindent
{\bf Conflict of interest} The authors declare that they have no conflict of interest.

\end{document}